\documentclass{amsart}[12pt]
\usepackage[dvipsone]{graphicx}
\usepackage{epsfig,amsmath,amssymb,enumerate}
\usepackage[all]{xy}

\newcommand{\lga}{\longrightarrow}
\newcommand{\lgaf}{\longleftarrow}

\newcommand{\limdi}{\displaystyle \lim_{\longrightarrow}}
\newcommand {\Z}{\mathbb Z}
\newcommand {\R}{\mathbb R}
\newcommand {\N}{\mathbb N}

\newcommand {\F}{\mathbb F}

\newcommand {\sub}{\subset}

\newcommand{\liminv}{\displaystyle\lim_{\longleftarrow}}

\newtheorem{theorem}{Theorem}[section]
\newtheorem{lemma}[theorem]{Lemma}

\newtheorem{proposition}[theorem]{Proposition}
\newtheorem{corollary}[theorem]{Corollary}

\theoremstyle{definition}
\newtheorem{definition}[theorem]{Definition}

\theoremstyle{remark}
\newtheorem{remark}[theorem]{Remark}

\numberwithin{equation}{section}

\begin{document}

\title{A topological equivalence relation for finitely presented groups}


\author{M. C\'ardenas}
\address{Departamento de Geometr\'{\i}a y Topolog\'{\i}a, Fac. Matem\'aticas, Universidad de Sevilla,
C/. Tarfia s/n 41012-Sevilla, Spain} \email{mcard@us.es}
\author{F. F. Lasheras}
\address{Departamento de Geometr\'{\i}a y Topolog\'{\i}a, Fac. Matem\'aticas, Universidad de Sevilla,
C/. Tarfia s/n 41012-Sevilla, Spain} \email{lasheras@us.es}
\author{A. Quintero}
\address{Departamento de Geometr\'{\i}a y Topolog\'{\i}a, Fac. Matem\'aticas, Universidad de Sevilla,
C/. Tarfia s/n 41012-Sevilla, Spain} \email{quintero@us.es}
\author{R. Roy}
\address{College of Arts and Sciences, New York Institute of Technology, Old Westbury, NY 11568-8000, USA}
\email{rroy@nyit.edu}

\thanks{This work was partially supported by the project MTM 2015-65397}


\date{December 1, 2019}


\begin{abstract}  In this paper, we consider an equivalence relation within
the class of finitely presented discrete groups attending to their asymptotic
topology rather than their asymptotic geometry. More precisely, we
say that two finitely presented groups $G$ and $H$ are ``proper
$2$-equivalent" if there exist (equivalently, for all) finite
$2$-dimensional CW-complexes $X$ and $Y$, with $\pi_1(X) \cong G$ and $\pi_1(Y) \cong H$, so that their universal covers $\widetilde{X}$ and
$\widetilde{Y}$ are proper $2$-equivalent.
It follows that this relation is coarser than the quasi-isometry
relation. We point out that finitely presented groups which are $1$-ended and semistable at
infinity are classified, up to proper $2$-equivalence, by their fundamental pro-group, and we study the behaviour
of this relation with respect to some of the main constructions in combinatorial group theory.
A (finer) similar equivalence relation may also be considered for groups of type $F_n, n \geq 3$, which
captures more of the large-scale topology of the group. Finally, we pay special attention to the class of those groups
$G$ which admit a finite $2$-dimensional CW-complex $X$ with $\pi_1(X) \cong G$ and whose universal cover
$\widetilde{X}$ has the proper homotopy type of a $3$-manifold. We show that if such a group $G$ is $1$-ended and
semistable at infinity then it is proper $2$-equivalent to either ${\Z} \times {\Z} \times {\Z}$, ${\Z} \times {\Z}$ or
${\F}_2 \times {\Z}$ (here, ${\F}_2$ is the free group on two generators). As it turns out, this applies in particular to any group $G$ fitting as
the middle term of a short exact sequence of infinite finitely presented groups, thus classifying such group extensions up to proper $2$-equivalence.
\end{abstract}

\maketitle

\section{Introduction}

It is well-known that an algebraic classification of finitely
generated groups is impossible because of the undecidability of the
word problem \cite{N}. In \cite{Gr}, Gromov outlined a program to
understand and try to classify these groups geometrically via the
notion of quasi-isometry for finitely generated groups, regarded as
metric spaces. Since then, those properties of finitely generated
groups which are invariant under quasi-isometries have been of great
interest and widely studied. On the other hand, the study of
asymptotic invariants of topological nature for finitely generated
groups has also led to an interesting research area (see \cite{Geo}
for a good source on this subject). See also \cite{O} for a nice survey
on
some of such invariants coming from $3$-manifold theory.\\
\indent In this paper, we consider an equivalence relation within
the class of finitely presented groups attending to their asymptotic topology
rather than their asymptotic geometry, and point out that this equivalence relation
is coarser than the quasi-isometry relation, i.e., quasi-isometric finitely presented
groups will also be related in this wider  and ``geometry forgetful" sense.
For this, we need some preliminaries.\\
\indent We will generally be working within the category of locally finite
CW-complexes and proper maps. We recall that a {\it proper} map is a
map with the property that the inverse image of every compact subset
is compact. Thus, two locally finite CW-complexes are said to be
proper homotopy equivalent if they are homotopy equivalent and all
homotopies involved are proper.\\
\indent Given a non-compact (strongly) locally finite CW-complex $Y$, a {\it
proper ray} in $Y$ is a proper map $\omega : [0, \infty) \lga Y$ (see \cite{Geo}). We
say that two proper rays $\omega, \omega'$ {\it define the same end}
if their restriction to the natural numbers $\omega|_{\N},
\omega'|_{\N}$ are properly homotopic. This equivalence relation
gives rise to the notion of {\it end determined by $\omega$} as the
corresponding equivalence class, as well as the space of ends ${\mathcal E}(Y)$ of $Y$ as a compact totally
disconnected metrizable space (see \cite{BQ,Geo}). The CW-complex $Y$ is {\it
semistable} at the end determined by $\omega$ if
any other proper ray defining the same end is in fact properly
homotopic to $\omega$; equi\-va\-lently,
if the fundamental pro-group $pro-\pi_1(Y, \omega)$ is pro-isomorphic to
a tower of groups with surjective bonding homomorphisms. Recall that $pro-\pi_1(Y, \omega)$ is represented by the inverse
sequence (tower) of groups
\[
\pi_1(Y, \omega(0))
\stackrel{\phi_1}{\lgaf} \pi_1(Y - C_1, \omega(t_1))
\stackrel{\phi_2}{\lgaf} \pi_1(Y - C_2, \omega(t_2))
\lgaf \cdots
\]
where $C_1
\sub C_2 \sub \cdots \sub Y$ is a filtration of $Y$ by compact subspaces, $\omega([t_i,\infty)) \sub Y - C_i$
and the bonding homomorphisms $\phi_i$ are
induced by the inclusions and basepoint-change isomorphisms (which are defined using subpaths of $\omega$). We refer to \cite{Geo,MarSe}
for more details and the basics of the pro-category of towers of groups.\\
\indent Given a CW-complex $X$, with $\pi_1(X) \cong G$, we will
denote by $\widetilde{X}$ the universal cover of $X$, constructed as
prescribed in (\cite{Geo}, $\S 3.2$), so that $G$ is acting freely
on the CW-complex $\widetilde{X}$ via a cell-permuting left action
with $G \backslash \widetilde{X} = X$. The number of ends of an
(infinite) finitely generated group $G$ represents the number of ends
of the (strongly) locally finite CW-complex $\widetilde{X}^1$, for some (equivalently any)
CW-complex $X$ with $\pi_1(X) \cong G$ and with finite $1$-skeleton,
which is either $1, 2$ or $\infty$ (finite groups have $0$ ends
\cite{Geo,SWa}). If $G$ is finitely presented, then $G$ is {\it
semistable at each end} (resp. {\it at infinity}, if $G$ is
$1$-ended) if the (strongly) locally finite CW-complex $\widetilde{X}^2$ is so, for some (equivalently any)
CW-complex $X$ with $\pi_1(X) \cong G$ and with finite $2$-skeleton.
In fact, we will refer to the fundamental pro-group of
$\widetilde{X}^2$ (at each end) as the fundamental pro-group of $G$
(at each end). Observe that any finite-dimensional locally finite CW-complex is strongly locally finite, see \cite[Prop. 10.1.12]{Geo}. It is worth mentioning that the number of ends and semistabilitiy are quasi-isometry invariants for finitely presented groups \cite{Br}.\\

\indent  In this context, we have:
\begin{definition} \label{p2e} Two finitely presented groups $G$ and $H$ are {\it proper $2$-equivalent}
if there exist (equivalently, for all) finite $2$-dimensional CW-complexes
$X$and $Y$, with $\pi_1(X) \cong G$ and $\pi_1(Y) \cong H$, so that their universal covers
$\widetilde{X}$ and $\widetilde{Y}$ are proper $2$-equivalent.
\end{definition}
See $\S 2$ for the definition of a proper $n$-equivalence between locally finite CW-complexes. Observe that any two finite groups are proper $2$-equivalent as any two simply connected finite CW-complexes are (trivially) proper $2$-equivalent, since they are homotopy equivalent to a finite bouquet of $2$-spheres (see \cite{Wall}). Also, the required proper $2$-equivalence can be replaced by a proper homotopy equivalence after wedging with $2$-spheres. We will give the details in $\S 2$ and show that this determines an equivalence relation for finitely presented groups. One can easily check that the number of ends
and semistability at each end are invariants under this proper $2$-equivalence relation, as well as any other proper
homotopy invariant of the group which depends only on the $2$-skeleton of the universal cover
of some finite CW-complex with the given group as fundamental
group, like simple connectivity at infinity (see \cite{Geo}). Likewise, the second cohomology $H^2(G; {\Z}G)$ of a
finitely presented group $G$ is an invariant under proper $2$-equivalences, as it is isomorphic to the first cohomology
of the end $H^1_e(\widetilde{X};{\Z})$, for any finite $2$-dimensional CW-complex $X$ with $\pi_1(X) \cong G$ (see
\cite[$\S 12.2, 13.2$]{Geo}). Also, observe that if $G$ is semistable at infinity then the second cohomology with
compact supports $H^2_c(\widetilde{X};{\Z}) \cong H^2(G;{\Z}G) \oplus (free \; abelian)$ is free abelian (and hence so
is $H^2(G;{\Z}G)$), see \cite{GM}. It is worth mentioning that the cohomology group $H^2(G;{\Z}G)$ was shown in
\cite{Ger} to be a quasi-isometry invariant of the group (see also \cite[Thm. 18.12.11]{Geo}).\\
\indent This proper $2$-equivalence
relation is motivated by the study and recent results \cite{CLQR_AMS} on {\it
properly $3$-realizable} groups, i.e., those finitely presented
groups $G$ for which there is some finite $2$-dimensional
CW-complex $X$ with $\pi_1(X) \cong G$ and whose universal cover $\widetilde{X}$ has the proper homotopy type of a
$3$-manifold. We will dedicate special attention to this class of groups in $\S \ref{App2}$ taking a closer look to it
regarding the above equivalence relation, and show that there are only three proper $2$-equivalence classes containing
$1$-ended and semistable at infinity properly $3$-realizable groups.\\
\indent It follows from \cite[Thm. 18.2.11]{Geo} that if $G$ and $H$ are in fact quasi-isometric groups then they are
also proper $2$-equivalent, even if the required proper $2$-equivalen\-ce in Definition \ref{p2e} is replaced by a proper homotopy equivalence (see \cite[Cor. 1.2]{CLQR_AMS}). On the other hand, it is easy to find finitely
presented groups which are trivially proper $2$-equivalent but not quasi-isometric. For this, we may just consider the
fundamental groups of the torus and any other closed orientable surface of genus at least $2$, which both have ${\R}^2$
as universal cover. A non-trivial example will be given in $\S \ref{counterexample}$.\\
\indent We also show that two finite graph of
groups decompositions with finite edge groups and finitely presented vertex groups with at most one end yield proper
$2$-equivalent groups if they have the same set of proper $2$-equivalence classes of vertex groups (see Theorem
\ref{graph}). On the other hand, unlike the situation under the quasi-isometry relation (compare with \cite[Thm.
0.4]{PaWhy}), the proper $2$-equivalence class of a finitely presented group does not determine in general the set of
proper $2$-equivalence classes of vertex groups in such a decomposition of the group. Again, an example will be given in
$\S \ref{counterexample}$.

\section{Some basic properties. The finite ended case} \label{basics}

We start by recalling the notions of proper $n$-equivalence and proper $n$-type for CW-complexes, already existing in the literature.
\begin{definition} (\cite[$\S$ 11.1]{Geo}) \label{AQ1}. A proper cellular map $f : X \lga Y$
between finite-dimensional locally finite CW-complexes is a proper
$n$-equivalence if there is another proper cellular map $g : Y \lga
X$ such that the restrictions $g \circ f | X^{n-1}$ and $f \circ g |
Y^{n-1}$ are proper homotopic to the inclusion maps  $X^{n-1}
\subseteq X$ and $Y^{n-1} \subseteq Y$.
\end{definition}
It is worth mentioning that an $n$-equivalence as in Definition \ref{AQ1} is a stronger version
of the notion of a proper $(n-1)$-type which extends the classical
$(n-1)$-type in ordinary homotopy theory introduced by J.H.C. Whitehead \cite{Wh}.
More precisely,
\begin{definition} (\cite[$\S$ 6]{CLMQ}) \label{AQ2} Let $X$ and $Y$ be two finite-dimensional locally finite CW-complexes.
We say that they have the same proper $n$-type if there exist proper
maps $f : X^{n+1} \lga Y^{n+1}$ and $g : Y^{n+1} \lga X^{n+1}$ such
that the restrictions $g \circ f | X^n$ and $f \circ g | Y^n$ are
proper homotopic to the inclusion maps $X^n \subseteq X^{n+1}$ and
$Y^n \subseteq Y^{n+1}$.
\end{definition}
\begin{remark} \label{dim} By the proper cellular approximation theorem (see \cite{Geo}), if two finite-dimensional
locally finite CW-complexes are proper $n$-equivalent then so are their $n$-skeleta, and two $n$-dimensional locally
finite CW-complexes are proper $n$-equivalent if and only if they have the same proper $(n-1)$-type. On the other hand,
if two finite-dimensional locally finite CW-complexes are proper homotopy equivalent then they are proper
$n$-equivalent, for all $n$.
\end{remark}
\begin{remark} Recently, Geoghegan et al. \cite{GeoGuilMih} have relaxed Definition \ref{AQ1}, making Definitions \ref{AQ1} and \ref{AQ2} more compatible.
\end{remark}
The following result shows that Definition \ref{p2e} does not depend on the choice of the corresponding CW-complexes. It also provides an alternative equivalent definition by replacing the required proper $2$-equivalence with a proper homotopy equivalence after wedging with $2$-spheres.
\begin{theorem} \label{alternative} Let $G$ and $H$ be two infinite finitely presented groups, and let $X$ and $Y$ be any two finite $2$-dimensional CW-complexes with $\pi_1(X) \cong G$ and $\pi_1(Y) \cong H$. Then, if $\widetilde{X}$ and $\widetilde{Y}$ denote the corresponding universal covers, the following statements are equivalent:
\begin{enumerate}
\item[(a)] The groups $G$ and $H$ are proper $2$-equivalent.
\item[(b)] $\widetilde{X}$ and $\widetilde{Y}$ are proper $2$-equivalent (i.e., they have the same proper $1$-type).
\item[(c)] There exist $2$-spherical objects $S^2_\alpha$ and $S^2_\beta$ so that $\widetilde{X} \vee S^2_\alpha$ and $\widetilde{Y} \vee S^2_\beta$ are proper homotopy equivalent.
\item[(d)] The universal covers $\widetilde{X \vee S^2}$ and $\widetilde{Y \vee S^2}$ are proper homotopy equivalent.
\end{enumerate}
\end{theorem}
\noindent The wedge in (d) is the usual one, while the wedges in (c) are taken along maximal trees $T \sub \widetilde{X}$ and $T' \sub \widetilde{Y}$, and by an spherical object we mean the space
obtained from the corresponding maximal tree by attaching finitely many $2$-spheres at each vertex.
\begin{proof}  We may always assume that the $0$-skeleta of the CW-complexes $X$ and $Y$ reduce to a single vertex. If $G$ and $H$ are proper $2$-equivalent then there exist finite $2$-dimensional CW-complexes $W$ and $Z$, with $\pi_1(W) \cong G$ and $\pi_1(Z) \cong H$, so that the universal
covers $\widetilde{W}$ and $\widetilde{Z}$ are proper $2$-equivalent. We now consider $K(G,1)$-complexes $X'$ and $W'$ with $(X')^2 = X$ and $(W')^2 = W$, and $K(H,1)$-complexes
$Y'$ and $Z'$ with $(Y')^2 = Y$ and $(Z')^2 = Z$. By the proper cellular approximation theorem, it is not hard to check
that $\widetilde{(X')^2} = \widetilde{X}$ and $\widetilde{(W')^2} = \widetilde{W}$ are proper $2$-equivalent as $X'$ and
$W'$ are homotopy equivalent. It also follows from \cite[Thm. 18.2.11]{Geo}, since $\pi_1(X') \cong \pi_1(W') \cong G$.
Similarly, $\widetilde{Y} = \widetilde{(Y')^2}$ and $\widetilde{Z} = \widetilde{(Z')^2}$ are proper $2$-equivalent. Thus, (a) $\Rightarrow$ (b) follows by transitivity. On the other hand, (b) $\Rightarrow$ (d) follows from the proof of \cite[Cor. 1.2]{CLQR_AMS}, and (d) $\Rightarrow$ (c) follows from the fact that $\widetilde{X \vee S^2} = \widetilde{X} \vee S^2_T$ where $S^2_T$ is the ``universal" spherical object defined by attaching exactly one $2$-sphere at each vertex of $T \sub \widetilde{X}$; similarly, $\widetilde{Y \vee S^2} = \widetilde{Y} \vee S^2_{T'}$. Conversely, for (c) $\Rightarrow$ (d) we observe that, by the classification of spherical objects in \cite[Prop. II.4.5]{BQ}, we have a proper homotopy equivalence
\[
\widetilde{X} \vee S^2_T \simeq \widetilde{X} \vee (S^2_\alpha \vee S^2_T) = (\widetilde{X} \vee S^2_\alpha) \vee S^2_T
\]
If $f : \widetilde{X} \vee S^2_\alpha \lga \widetilde{Y} \vee S^2_\beta$ is now a (cellular) proper homotopy equivalence, then the restriction $\omega = f|_T$ yields a proper homotopy equivalence $(\widetilde{X} \vee S^2_\alpha) \vee S^2_T \simeq (\widetilde{Y} \vee S^2_\beta) \vee S^2_\omega$, where $S^2_\omega$ is the spherical object obtained by attaching $\# \omega^{-1}(v)$ $2$-spheres at each vertex $v \in T'$. As $\omega$ induces a homeomorphism between the spaces of ends ${\mathcal E}(T) \simeq {\mathcal E}(\widetilde{X})$ and ${\mathcal E}(T') \simeq {\mathcal E}(\widetilde{Y})$, the classification of spherical objects in \cite[Prop. II.4.5]{BQ} yields a proper homotopy equivalence $S^2_\omega \simeq S^2_{T'}$ and hence
\[
(\widetilde{Y} \vee S^2_\beta) \vee S^2_\omega \simeq (\widetilde{Y} \vee S^2_\beta) \vee S^2_{T'} \simeq \widetilde{Y} \vee (S^2_\beta \vee S^2_{T'}) \simeq \widetilde{Y} \vee S^2_{T'}
\]
thus obtaining a proper homotopy equivalence $\widetilde{X \vee S^2} = \widetilde{X} \vee S^2_T \simeq \widetilde{Y} \vee S^2_{T'} = \widetilde{Y \vee S^2}$. Finally, (d) $\Rightarrow$ (a) is obvious as $\pi_1(X \vee S^2) \cong G$ and $\pi_1(Y \vee S^2) \cong H$.
\end{proof}
\begin{remark} Theorem \ref{alternative}(d) shows that the alternative definition via proper homotopy equivalences does not depend on the choice of the corresponding CW-complexes after taking wedge with a single $2$-sphere.
\end{remark}
\begin{corollary} The relation of being proper $2$-equivalent is an equivalence relation for finitely presented groups.
\end{corollary}
\begin{proof} It readily follows from the transitivity of proper $n$-equivalences for CW-complexes. Alternatively, in view of Theorem \ref{alternative}, one can also show transitivity as follows. Let $G, H$ and $K$ be infinite finitely presented groups so that $G$ is proper
$2$-equivalent to $H$ and $H$ is proper $2$-equivalent to $K$, and let
$X, Y$ and $Z$ be finite $2$-dimensional CW-complexes with $\pi_1(X) \cong G, \pi_1(Y) \cong H$ and $\pi_1(Z) \cong K$.
By Theorem \ref{alternative}, we have that $\widetilde{X \vee S^2}$ is proper homotopy equivalent to $\widetilde{Y \vee
S^2}$ which in turn is proper homotpy equivalent to $\widetilde{Z \vee S^2}$. Thus, $G$ and $K$ are proper
$2$-equivalent, as $\pi_1(X \vee S^2) \cong G$ and $\pi_1(Z \vee S^2) \cong K$.
\end{proof}
\begin{remark} \label{quasi-isometric}
In particular, if $G$ and $H$ are two quasi-isometric finitely presented groups and $X$ and $Y$ are finite
$2$-dimensional CW-complexes with $\pi_1(X) \cong G$ and $\pi_1(Y) \cong H$, then it follows from \cite[Thm.
18.2.11]{Geo} that their universal covers $\widetilde{X}$ and $\widetilde{Y}$ are proper $2$-equivalent (in fact, $\widetilde{X \vee S^2}$ and $\widetilde{Y \vee S^2}$ are proper homotopy equivalent, by Theorem \ref{alternative}) and hence $G$ and $H$ are proper $2$-equivalent as finitely presented groups. By using the \v{S}varc-Milnor Lemma and its well-known
corollaries (see \cite[Thm. 18.2.15]{Geo} or \cite{DructuKapovich}) we have as an immediate consequence that if $H \leq G$ has
$[G:H] < \infty$ and $N \leq G$ is a finite normal subgroup then $G, H$ and $G/N$ are proper $2$-equivalent to each
other. In particular, all $2$-ended groups are proper $2$-equivalent to the group of integers.
\end{remark}

Notice that we may define a function $\varphi$ between the set of proper $2$-equivalence classes of finitely presented
groups and the set of proper $1$-types of $2$-dimensional (locally finite) simply connected CW-complexes, by assigning
to each equivalence class $[G]$ the proper $1$-type of $\widetilde{X}$, where $X$ is any finite $2$-dimensional
CW-complex with $\pi_1(X) \cong G$. Of course, all spaces with $k$ ends where $3 \leq k < \infty$ are not in the image of $\varphi$. Even when considering only simply connected CW-complexes with an allowed number of ends, this function $\varphi$ is not expected to be surjective, as suggested by the
following example. First, consider the mapping telescope of the obvious map $S^1 \lga S^1$ of degree two and cut it in
half through a transverse circle $C$. Second, consider the right hand side of it and glue a disk to the
circle $C$ along the boundary, via the identity map (see the figure below). The second cohomology with compact supports
of the resulting $2$-dimensional simply connected CW-complex is not free abelian (as contains the dyadics as a subgroup)
and its fundamental pro-group has the form $1 \leftarrow {Z} \stackrel{\times 2}{\lgaf} {\Z} \stackrel{\times 2}{\lgaf}
{\Z} \stackrel{\times 2}{\lgaf} \cdots$

\begin{figure}[h!]
\vspace{-3mm}
\centerline{\psfig{figure=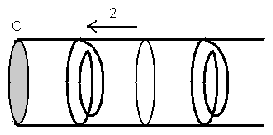,height=2.5cm,width=9cm}}
\end{figure}

If $\varphi$ were surjective, then there would exist a $1$-ended finitely presented group $G$ which is not semistable at
infinity (as $H^2(G;{\Z}G)$ is not free abelian, see $\S 1$) and would also serve as a counterexample to the Burnside
problem for finitely presented groups (as no free ${\Z}$-action is possible with such a fundamental pro-group, see
\cite{GeoG}). This leads us to the following open question
(some related questions have already been posed in \cite{GeoG}):\\

\noindent {\bf Open question:} What is the image of $\varphi$?\\

\indent Observe that given two (infinite) finitely presented proper $2$-equivalent groups $G$ and $H$ which are
semistable at each end, and finite $2$-dimensional CW-complexes $X$ and $Y$ with $\pi_1(X) \cong G$ and $\pi_1(Y) \cong
H$, for any end of $\widetilde{X}$ there is an end of $\widetilde{Y}$ so that the corresponding fundamental pro-groups
are pro-isomorphic, and vice versa (see \cite[Thm. 16.2.3]{Geo}). In fact, in the $1$-ended case we have the following
characterization.
\begin{proposition} \label{characterisation} Let $G$ and $H$ be two finitely presented groups which are $1$-ended and
semistable at infinity. Then, $G$ and $H$ are proper $2$-equivalent if and only if they have pro-isomorphic fundamental
pro-groups.
\end{proposition}
Indeed, with the above notation, if $pro-\pi_1(\widetilde{X}) \cong pro-\pi_1(\widetilde{Y})$ (regardless the base ray,
by the semistability condition) then by \cite[Prop. 6.2]{CLMQ} we have that $\widetilde{X}$ and $\widetilde{Y}$ are
proper $2$-equivalent and hence $G$ and $H$ are proper $2$-equivalent. In particular,
all $1$-ended and simply connected at infinity (finitely presented) groups are proper $2$-equivalent to each other and
hence proper $2$-equivalent to ${\Z} \times {\Z} \times {\Z}$.

\indent We conclude this section by studying the behaviour of the proper $2$-equivalence relation with respect to direct
products and group extensions in general.
\begin{lemma} \label{complex-product} Let $X, Y, Z$ and $W$ be finite-dimensional locally finite CW-complexes, and
assume $X$ and $Y$ are proper $n$-equivalent to $Z$ and $W$ respectively.Then, $X \times Y$ is proper $n$-equivalent to
$Z \times W$.
\end{lemma}
\begin{proof} Let $f : X \lga Z$ and $g : Y \lga W$ be proper $n$-equivalences together with proper cellular maps $f' :
Z \lga X$ and $g' : W \lga Y$ and proper homotopies $F : f' \circ f | X^{n-1} \simeq i_{X^{n-1}}$, $F' : f \circ f' |
Z^{n-1} \simeq i_{Z^{n-1}}$, $G : g' \circ g | Y^{n-1} \simeq i_{Y^{n-1}}$ and $G' : g \circ g' | W^{n-1} \simeq
i_{W^{n-1}}$. By the proper cellular approximation theorem, we may also assume that all those proper homotopies are
cellular. Therefore, the product maps $f \times g$ and $f' \times g'$ as well as the homotopies $H : X^{n-1} \times
Y^{n-1} \times I \lga X \times Y$ and $H' : Z^{n-1} \times W^{n-1} \times I \lga Z \times W$, given by $H(x,y,t) =
(F(x,t), G(y,t))$ and $H'(z,w,t) = (F'(z,t), G'(w,t))$, are all cellular maps. Thus, as $(X \times Y)^k \subseteq X^k
\times Y^k$ for all $k \geq 0$, the restrictions $H | (X \times Y)^{n-1} \times I : (f' \times g') \circ (f \times g) |
(X \times Y)^{n-1} \simeq i_{(X \times Y)^{n-1}}$ and $H' | (Z \times W)^{n-1} \times I : (f \times g) \circ (f' \times
g') | (Z \times W)^{n-1} \simeq i_{(Z \times W)^{n-1}}$ are the desired proper homotopies, and this concludes the proof.
\end{proof}
\begin{proposition} \label{group-product} Let $G, G', H$ and $H'$ be finitely presented groups and assume that $G$ and
$H$ are proper $2$-equivalent to $G'$ and $H'$ respectively. Then, $G \times H$ is proper $2$-equivalent to $G' \times
H'$.
\end{proposition}
\begin{proof} First, if either $G$ (and hence $G'$) or $H$ (and hence $H'$) is finite then $G \times H$ and $G' \times
H'$ have proper $2$-equivalent finite index subgroups, and hence we are done by Remark \ref{quasi-isometric}, as any two finite groups are proper $2$-equivalent. Otherwise,
there exist finite $2$-dimensional CW-complexes $X, X', Y$ and $Y'$ having $G, G', H$
and $H'$ as fundamental groups and so that $\widetilde{X}$ and $\widetilde{X'}$ are proper $2$-equivalent to
$\widetilde{Y}$ and $\widetilde{Y'}$ respectively. By Lemma \ref{complex-product} above, $\widetilde{X} \times
\widetilde{Y} = \widetilde{X \times Y}$ is then proper $2$-equivalent to $\widetilde{X'} \times \widetilde{Y'} =
\widetilde{X' \times Y'}$ and hence so they are their $2$-skeleta, by Remark \ref{dim}. The conclusion follows, since $\displaystyle \left( \widetilde{X \times Y} \right)^2 = \widetilde{(X \times
Y)^2}$ and $\displaystyle \left( \widetilde{X' \times Y'} \right)^2 = \widetilde{(X' \times Y')^2}$.
\end{proof}
More generally, Proposition \ref{group-product} together with \cite[Thm. 16.8.4]{Geo} yield the following result.
\begin{proposition} \label{group-extension} Let $1 \rightarrow G \lga K \lga H \rightarrow 1$ and $1 \rightarrow G' \lga
K' \lga H' \rightarrow 1$ be two short exact sequences of finitely presented groups, and assume that $G$ and $H$ are
proper $2$-equivalent to $G'$ and $H'$ respectively. Then, $K$ is proper $2$-equivalent to $K'$. In fact, $K$ and $K'$
are proper $2$-equivalent to $G \times H$ and $G' \times H'$ respectively.
\end{proposition}
\begin{proof} If either $G$ (and hence $G'$) or $H$ (and hence $H'$) is finite, then the conclusion easily follows from
Remark \ref{quasi-isometric}. Otherwise, by \cite[Thm. 16.8.4]{Geo}, given finite $2$-dimensional CW-complexes  $X, X',
Y$ and $Y'$ having $G, G', H$ and $H'$ as fundamental groups, there exist finite CW-complexes $Z$ and $Z'$ having $K$
and $K'$ as fundamental groups and so that their universal covers $\widetilde{Z}$ and $\widetilde{Z'}$ are proper
$2$-equivalent to $\widetilde{X} \times \widetilde{Y}$ and $\widetilde{X'} \times \widetilde{Y'}$ respectively (and
hence so are their $2$-skeleta). Therefore, $K$ and $K'$ are proper $2$-equivalent to $G \times H$ and $G' \times H'$
respectively, and the conclusion follows then from Proposition \ref{group-product}.
\end{proof}
Notice that if the groups $G, G', H$ and $H'$ in Propositions \ref{group-product} and \ref{group-extension} above are
all infinite then the corresponding products (resp. group extensions) are $1$-ended and semistable at infinity \cite{M}
and have pro-isomorphic fundamental pro-groups, by Proposition \ref{characterisation}. It is worth mentioning that they
are all of telescopic type at infinity (see $\S 5$), i.e., their fundamental pro-groups are always pro-free and
pro-(finitely generated) \cite{CLR, CLQRoy}; in fact, they are all proper $2$-equivalent to one of the following groups
${\Z} \times {\Z} \times {\Z}$, ${\Z} \times {\Z}$ or ${\F}_2 \times {\Z}$, as we will show in $\S \ref{App2}$.

\section{The infinite ended case}

Since the quasi-isometry relation is stronger than the proper $2$-equivalence relation for finitely presented groups,
the following result follows from \cite[Lemma 1.4]{PaWhy} together with Remark \ref{quasi-isometric}.
\begin{lemma} \label{power} Let $G$ be an infinite finitely presented group, and $n \geq 2$. Then, $G*
\stackrel{(n)}{\cdots} *G$ is proper $2$-equivalent to $G*G$ (in fact, they are quasi-isometric).
\end{lemma}
\begin{remark} Alternatively, an intuitive direct proof could be roughly as follows. Given any finite $2$-dimensional
CW-complex $X$ with $\pi_1(X) \cong G$ and denoting by $X_n = X \vee \stackrel{(n)}{\cdots} \vee X$ ($n \geq 2$) the
corresponding wedge, we claim that the universal covers $\widetilde{X_n}$ and $\widetilde{X_2}$ are already proper
homotopy equivalent. For this, observe that $\widetilde{X_n}$ consists of a tree-like arrangement of a collection of
copies of $\widetilde{X}$ in such a way that $n$ of those copies meet appropriately at each vertex of $\widetilde{X_n}$.
Given a vertex on a fixed copy of $\widetilde{X}$ inside $\widetilde{X_n}$, the idea is to start pushing somehow towards
infinity, i.e., outside some (increasing) finite subcomplex of $\widetilde{X_n}$ containing the given vertex, some of
the various copies of $\widetilde{X}$ we encounter within that finite subcomplex so as to keep only two of such copies
at each vertex, and keep repeating the argument to an increasing (finite) number of fixed copies of $\widetilde{X}$
inside $\widetilde{X_n}$. A limit process would then give us the desired proper homotopy equivalence. The missing
details are of the same sort and complexity as those shown in the argument for the proof of Lemma \ref{Antonio} and Proposition
\ref{free} below, and we omit them for the sake of simplicity.
\end{remark}
\begin{proposition} \label{free} Let $G, H$ and $K$ be infinite finitely presented groups, and assume that $G$ and $H$
are proper $2$-equivalent. Then, $G*K$ and $H*K$ are also proper $2$-equivalent.
\end{proposition}
The converse of this result does not hold in general, i.e., common free factors can not be canceled out and still obtain
a proper $2$-equivalence in general. A simple example would be ${\Z} * {\Z}$ versus ${\F}_2 * {\Z}$. We need the following technical lemma for the proof of Proposition \ref{free}.

\begin{lemma} \label{Antonio} Let $f : X \lga Y$ be a proper homotopy equivalence between two infinite
finite-dimensional locally finite (connected) CW-complexes. Then, $f$ is properly homotopic to a map $g : X \lga
Y$ which is a bijection $g : X^0 \lga Y^0$ between the $0$-skeleta. Moreover, given arbitrary vertices $v_0 \in X$ and
$w_0 \in Y$, the map $g$ can be chosen so that $g(v_0)=w_0$.
\end{lemma}
\begin{proof}
We may always assume that $f$ is a cellular map, by the Proper Cellular Approximation Theorem (see \cite{Geo}).
Moreover, given arbitrary vertices $v_0 \in X$ and $w_0 \in Y$, we may further assume that $f(v_0) = w_0$. Indeed, let
$\gamma: I \lga Y$ be an edge path with $\gamma(0) = f(v_0)$ and $\gamma(1) = w_0$. By applying the Proper Homotopy
Extension Property
to $f$ and the path $\gamma$, we get a proper map $\widetilde {f}: X \lga Y$ properly homotopic to $f$ with
$\widetilde{f}(v_0) = w_0$.\\
\indent Let $T_X \subset X$ and $T_Y \subset Y$ be maximal trees. By choosing $v_0 \in T_X$ as a root vertex we have the
usual ordering on the vertices of $X$ by setting $v \leq v'$ if $v$ lies in the unique path $\Gamma(v_0,v') \subset T_X$
from $v_0$ to $v'$. Moreover, we write $|v| = n$ if $\Gamma(v_0,v)$ contains exactly $n+1$ vertices. The integer $|n|$
is termed the height of $v$. More generally, given a subcomplex $Z\subset X$ we write $|Z| = \min \{|v| \; , v \in
Z^0\}$. Similarly, by fixing  $w_0 \in T_Y$ as a root vertex, we have the same kind of ordering on $Y^0$.\\
\indent We are now ready to prove the lemma. For this, we set $X^0(n) = \{v \in X^0 \; , |v| \leq n \}$ and let $T_X(n)
\subset T_X$ denote the finite subtree generated by the set $X^0(n)$. Similarly, we define $Y^0(n)$ and $T_Y(n)$, $n
\geq 1$. We first find a proper map $h : X \lga Y$ properly homotopic to $f$ and such that $h$ restricts to a surjection
$h : X^0 \lga Y^0$. The definition of the map $h$ will follow from the inductive construction of an increasing
subsequence $0= n_0 < n_1 < \dots < n_j < \dots$ and maps $h_j:  X^0(n_j) \lga Y^0$ satisfying the following properties:
\begin{enumerate}
\item [(a)] $h_j$ extends $h_{j-1}$.
\item [(b)] $Y^0(j) \subset h_j(X^0(n_j))$.
\item [(c)] $|\Gamma(f(v),h_j(v))| \geq \min \{|f(v)|, j\}-1$, if $|v| > n_{j-1}$.
\end{enumerate}
In order to construct the maps $h_j$, we start choosing a subsequence $m_1 < m_2 < \dots$ such that all components of
$cl(T_Y - T_Y(m_i))$ are unbounded. For the sake of simplicity, we assume that the tree $T_Y$ is such that we
can choose $m_i = i$ for all $i\geq 1$, easing the reading of what follows. This way, any vertex $w \in Y^0$ with $|w| =
j$ is either a terminal vertex or there are infinitely many vertices $w' \geq w$. In any case, we can find infinitely
many vertices with the property
\begin{equation}\label{F1}
|w'| \geq j+1 \mbox{ and } |\Gamma(w,w')|\geq j-1.
\end{equation}

We start simply by taking $n_0 =0$, and $h_0(v_0) = f(v_0) = w_0$.\\

\indent Assume $h_j$ already constructed. Let $L^Y_{k} = \{w \in Y^0 \; , |w| = k \mbox{ and } h_j^{-1}(w) = \emptyset
\}$. By induction, $k_j = \min \{k \; , L^Y_k \neq \emptyset\} \geq j+1$. For each $w \in L^Y_{k_j}$, we choose a vertex
$v_w\in X^0$ such that
$|v_w| > n_j$ and $|\Gamma(w,f(v_w)|\geq |w|-1 = k_j-1 \geq j$.
 Here we use (\ref{F1}) and the fact that $f$ induces a homeomorphism between the corresponding spaces of ends. Next, we
choose $n_{j+1} > n_j$ large enough to have
$\{v_w \; , w \in L^Y_{k_j}\} \cup X^0(n_j) \sub X^0(n_{j+1})$, and define $h_{j+1} : X^0(n_{j+1}) \lga Y^0$ by setting
$h_{j+1}(v) = h_j(v)$ if $v\in X^0(n_j)$,
$h_{j+1}(v_w) = w$ if $w\in L^Y_{k_j}$ and $h_{j+1}(v) = f(v)$ otherwise. Clearly, $h_{j+1}$ satisfies conditions (a)
and (b above). In order to verify (c), we observe that for any vertex $|v| > n_j$ we have $|\Gamma(h_{j+1}(v_w),
f(v_w))| = |\Gamma(w,f(v_w))| \geq |w|-1 \geq j$ and $|\Gamma(h_{j+1}(v), f(v))| = |f(v)|$ if $v \neq v_w$, for all $w
\in L^Y_{k_j}$.\\
\indent This way, the union $\displaystyle h^0 = \cup_j h_j : X^0 \lga Y^0$ is a proper surjection by
conditions (b) and (c) above. Furthermore, condition (c) yields that the family of arcs $\{\Gamma(f(v), h^0(v))\}_{v \in
X^0}$ is locally finite and hence the map $H^0 : X \times \{0\} \cup X^0 \times I \lga Y$ given by $H^0(x,0) = f(x)$ and
$H^0|\{v\} \times I = \Gamma(h^0(v),f(v))$ is proper. Thus, by the Proper Homotopy Extension Property, $H^0$ extends to
a proper homotopy $H: X \times I \lga Y$, and $h = H_1$ is the required map.\\
\indent Next, we will modify the map $h$ so as to get the desired proper map $g: X \lga Y$. This time, we will construct
injective maps $g_j : h^{-1}(Y^0(j)) \lga Y^0$ inductively with the following properties
\begin{enumerate}
\item [(a)] $g_j$ extends $g_{j-1}$.
\item [(b)] $Y^0(j) \subset g_j(h^{-1}(Y^0(j)))$.
\item [(c)] $|\Gamma(h(v),g_j(v))| \geq |h(v)|-1$, for all $v \in X^0$.
\end{enumerate}
since $h_0^{-1}(w_0) = \{v_0\}$, we simply set $g_0 = h_0: \{v_0\} \lga Y^0$. Assume $g_j$ already defined. Given $w \in
Y^0$ with $|w| = j+1$ and $h^{-1}(w) = \{v_0^w,v_1^w, \dots, v^w_{k(w)}\}$, we define $g_{j+1}$ on $h^{-1}(w)$ according
to the following two cases
\begin{enumerate}
\item There is some $v \in h^{-1}(Y^0(j))$ with $w = g_j(v)$. Then, we define the vertices $g_{j+1}(v_i^w)$ ($0 \leq i
\leq k(w)$) so as to form a set of $k(w)+1$ vertices in $Y^0 - g_j(h^{-1}(Y^0(j)))$ such that $|\Gamma
(g_{j+1}(v_i^w), g_j(v))| \geq |g_j(v)|-1$.
\item Assume $w$ is outside the image of $g_j$. Then, we define $g_{j+1}(v_0^w) = w$ and the vertices $g_{j+1}(v_i^w)$
($1\leq i \leq k(w)$) so as to form a set of
$k(w)$ vertices in $Y^0 - g_j(h^{-1}(Y^0(j)))$ with $|g_{j+1}(v_i^w)| > |w|$ and $|\Gamma(g_{j+1}(v_i^w),w)|
\geq |w|-1$.
\end{enumerate}
In both cases, we also choose the vertices in such a way that $g_{j+1}(v^w_i) \neq g_{j+1}(v_{i'}^{w'})$ for all $i\leq
k(w)$ and $i'\leq k(w')$ whenever $w\neq w'$. Notice that $g_{j+1}$ thus defined satisfies conditions (a) and (b) above.
Moreover, condition (c) also holds for $g_{j+1}$. This is obvious if $v \in h^{-1}(Y^0(j))$ by the induction hypothesis.
Otherwise, if $v \notin h^{-1}(Y^0(j))$ but $h(v) = g_j(v')$ for some $v'\in h^ {-1}(Y^0(j))$, we have
\[
|\Gamma(g_{j+1}(v),h(v))| = |\Gamma(g_{j+1}(v),g_j(v'))|\geq |g_j(v')| - 1 = |h(v)| - 1
\]
Finally, if $v \notin h^{-1}(Y^0(j))$ and $h(v)$ misses the image of $g_j$, then
\[
|\Gamma(g_{j+1}(v),h(v))|\geq |h(v)| - 1, \mbox{ for } h(v) = w \mbox{ as in } (2)
\]
\indent The map $g^0: X^0 \lga Y^0$ given by $g^0(v) = g_j(v)$ if $v \in h^{-1}(Y^0(j))$ is then injective, and hence a
bijection by property (b). Moreover, property (c) yields that the family of arcs $\{\Gamma(g^0(v),h(v))\}_{v \in X^0}$
is locally finite and hence we obtain a proper map $g: X \lga Y$ properly homotopic to $h$ by applying the Proper
Homotopy Extension Property to
$G^0 : X \times \{0\} \cup X^0 \times I \lga Y$ given by $G^0(x,0) = h(x)$ and $G^0|\{v\} \times I = \Gamma(g^0(v),h(v))$.
\end{proof}
\begin{proof}[Proof of Proposition \ref{free}] Let $X, Y$ and $Z$ be finite $2$-dimensional CW-complexes having $G, H$
and $K$ as fundamental groups and whose $0$-skeleta consists of a single vertex $X^0 = \{x_0\}, Y^0 =\{y_0\}$ and $Z^0 =
\{z_0\}$, together with a (cellular) proper homotopy equivalence $f : \widetilde{X} \lga \widetilde{Y}$ and base points
$\widetilde{x}_0 \in \widetilde{X}, \widetilde{y}_0 = f(\widetilde{x}_0) \in \widetilde{Y}$ and $\widetilde{z}_0 \in
\widetilde{Z}$. We may further assume that $f$ is a bijection between the $0$-skeleta, by Lema \ref{Antonio}. It is
clear that $\pi_1(X \vee Z) \cong G*K$ and $\pi_1(Y \vee Z) \cong H*K$. We will show that $G*K$ and $H*K$ are proper
$2$-equivalent by finding a proper homotopy equivalence between the universal covers $\widetilde{X \vee Z}$ and
$\widetilde{Y \vee Z}$. According to \cite{SWa}, these universal covers can be described as in \cite{ACLQ} as follows.\\

\indent (i) The universal cover $\widetilde{X \vee Z}$ is a tree-like arrangement of countably many copies
$\widetilde{X}_p$ and $\widetilde{Z}_r$ of $\widetilde{X}$ and $\widetilde{Z}$, whose vertices  are identified to get
the $0$-skeleton $\widetilde{X\vee Z}^0$ via a bijection $\alpha : {\N} \times {\N} \lga {\N} \times {\N}$ given by the
group action of $G*K$ on $\widetilde{X \vee Z}$; that is, by choosing a bijection for each $p$ and $r$, the
$0$-skeleta $\widetilde{X}^0_p$ and $\widetilde{Z}^0_r$ can be regarded as the sets $\{(p,q); q \in {\N}\}$ and $\{(r,s); s
\in {\N}\}$, respectively, and then $(p,q)$ gets identified with $\alpha(p,q)$ to obtain the $0$-skeleton $\widetilde{X
\vee Z}^0$.

\indent (ii) Similarly, the universal cover $\widetilde{Y \vee Z}$ is a tree-like arrangement of copies
$\widetilde{Y}_a$ and ${\widetilde{Z}'}_c$ of $\widetilde{Y}$ and $\widetilde{Z}$, whose vertices $\widetilde{Y}^0_a =
\{(a,b)\}_{b \in {\N}}$ and $\widetilde{Z'}^0_c = \{(c,d)\}_{d \in {\N}}$ are identified to get the $0$-skeleton
$\widetilde{Y\vee Z}^0$ according to a bijection $\alpha' : {\N} \times {\N} \lga {\N} \times {\N}$.\\

\indent We choose a ``root" copy $\widetilde{X}_{p_0} \sub \widetilde{X \vee Z}$ which gives the height $0$ for the
copies of $\widetilde{X}$. Then the copies of $\widetilde{Z}$ at height $0$  are the $\widetilde{Z}_r$'s for which
$\alpha(p_0,q) = (r,s)$ for some $q,s \in {\N}$. Now a copy $\widetilde{X}_p$ ($p \neq p_0$) is said to be at height $1$ if a
vertex in $\widetilde{X}_p$ is identified to a vertex of a copy of $\widetilde{Z}$ at height $0$. This way we can
define the height of any copy of $\widetilde{X}$ and $\widetilde{Z}$. Let $| \ |_1$ denote this height function. Similarly, by choosing a ``root" copy $\widetilde{Y}_{a_0} \sub \widetilde{Y \vee Z}$ we can define a height function $| \ |_2$ for the copies of $\widetilde{Y}_a$ and $\widetilde{Z}_c$.\\
\indent For each $k \geq 0$, let $\displaystyle L^k_1(X) = \bigsqcup_{|\widetilde{X}_p|_1 \leq k} \widetilde{X}_p$ and $\displaystyle L^k_1(Z) = \bigsqcup_{|\widetilde{Z}_r|_1 \leq
k} \widetilde{Z}_r$. Similarly, we define
$\displaystyle L^k_2(Y) = \bigsqcup_{|\widetilde{Y}_a|_2 \leq k} \widetilde{Y}_a$ and $\displaystyle L^k_2(Z) = \bigsqcup_{|{\widetilde{Z}'}_c|_2 \leq k} {\widetilde{Z}'}_c$.\\
\indent We are ready to define inductively (cellular) proper homotopy equivalences
\[
f_k: L^k_1(X) \lga L^k_2(Y) \mbox{ and } g_k: L^k_1(Z) \lga L^k_2(Z)
\]
such that $f_k$ and $g_k$ are extensions of $f_{k-1}$ and $g_{k-1}$, respectively, and $g_k \circ \alpha = \alpha' \circ
f_k$ for all $k\geq 0$. We start by considering cellular
homeomorphisms $\varphi: \widetilde{X}_{p_0} \lga \widetilde{X}$ and $\psi: \widetilde{Y}_{a_0} \lga \widetilde{Y}$  and
by Lemma \ref{Antonio} we choose $f_0: \widetilde{X}_{p_0} \lga \widetilde{Y}_{a_0}$ to be a (cellular) proper map
properly homotopic to $\psi^{-1}\circ f\circ \varphi$ which restricts to a bijection between the $0$-skeleta. In order
to define $g_0: L^0_1(Z)\lga L^0_2(Z)$, let $Z_r$ with $|Z_r|_1 = 0$. Then there exist exactly  two indices $s,q \in
{\N}$ such that $\alpha(p_0,q) = (r,s)$. Let $c \in {\N}$ be such that the copy ${\widetilde{Z}'}_c$ contains a
vertex $(c,d)$ for which $\alpha'(f_0(p_0,q)) = (c,d)$. Then we apply Lemma \ref{Antonio} to any
cellular homeomorphism $Z_r\cong Z_c$ to get a (cellular) proper homotopy equivalence $g_{0,r}: \widetilde{Z}_r \lga
{\widetilde{Z}'}_c$ which carries the vertex $(r,s) \in \widetilde{Z}_r$ to the vertex $(c,d) \in {\widetilde{Z}'}_c$. The
union map of all $g_{0,r}$ defines a (cellular) proper homotopy equivalence $g_0: L^0_1(Z) \lga L^0_2(Z)$ which
restricts to a bijection between $0$-skeleta. Moreover, by definition $g_0 \circ \alpha = \alpha' \circ f_0$.\\
\indent Assume that we have already defined the maps $f_k$ and $g_k$. For any $\widetilde{X}_p$ with $|\widetilde{X}_p|_1 =
k+1$ we find exactly a vertex $(p,q) \in \widetilde{X}_p$ such that the vertex $(r,s) = \alpha(p,q)$ belongs to a copy
$\widetilde{Z}_r$ with $|\widetilde{Z}_r|_1 = k$. Let $\widetilde{Y}_a$ be the only copy at height $k+1$ for which there
exists a vertex $(a,b) \in \widetilde{Y}_a$ with $\alpha'(a,b) = g_k(r,s)$.
Then if $\varphi: \widetilde{X}_p \cong \widetilde{X}$ and $\psi: \widetilde{Y}_a \cong \widetilde{Y}$ are (cellular)
homeomorphisms we apply Lemma \ref{Antonio} to the composite $\psi^{-1}\circ f\circ \varphi$ to get a proper homotopy
equivalence $f_{k+1,p}: \widetilde{X}_p \lga \widetilde{Y}_a$ restricting to a bijection between the $0$-skeleta and
$f_{k+1,p}(p,q) = (a,b)$. The union of the maps $f_k$ and $f_{k+1,p}$ yields a proper homotopy equivalence $f_{k+1}:
L^{k+1}_1(X) \lga L^{k+1}_2(Y)$ which extends $f_k$. A similar argument can be applied to define $g_{k+1}$ as an extension of $g_k$.\\
\indent Once the $f_k$'s and $g_k$'s have been defined, the obvious filtered maps $f: \bigsqcup_{p \in {\N}}
\widetilde{X}_p \lga \bigsqcup_{a \in {\N}} \widetilde{Y}_a$ and $g: \bigsqcup_{r \in {\N}} \widetilde{Z}_r
\lga \bigsqcup_{c \in {\N}} {\widetilde{Z}'}_c$ defined by $f|L^k_1(X) = f_k$ and $g|L^k_1(Z) = g_k$  turn to be
(cellular) proper homotopy equivalences. Moreover, the following diagram commutes

\[
\xymatrix{
\ar[d]^{f} \displaystyle \bigsqcup_{p \in {\N}} \widetilde{X}_p &
\ar[l]_{i} {\N} \times {\N} \ar[d]^{f_0} \ar[r]^{\alpha} & {\N} \times {\N} \ar[d]^{g_0} \ar[r]^{j} & \displaystyle
\bigsqcup_{r \in {\N}} \widetilde{Z}_r \ar[d]^{g}\\
\displaystyle \bigsqcup_{a \in {\N}} \widetilde{Y}_a  & \ar[l]_{i'} {\N} \times {\N} \ar[r]^{\alpha'} & {\N} \times {\N}
\ar[r]^{j'} & \displaystyle \bigsqcup_{c \in {\N}} {\widetilde{Z}'}_c\\
}
\]
where the maps $i$, $i'$, $j$, $j'$ denote the inclusions of $0$-skeleta and $f_0$ and $g_0$ are the corresponding
restrictions.
Furthermore, $\widetilde{X \vee Z}$ and $\widetilde{Y \vee Z}$ are the pushouts of the first and second row, respectively.
Thus, by the gluing lemma \cite[Lemma I.4.9]{BQ}, the map $f \cup g : \widetilde{X \vee Z} \lga \widetilde{Y \vee Z}$
obtained by the pushout construction is a proper homotopy equivalence.
\end{proof}
\begin{corollary} Let $G, H$ and $K$ be finitely presented groups. If $G$ and $H$ are proper $2$-equivalent and both
$G*K$ and $H*K$ are infinite ended, then $G*K$ and $H*K$ are also proper $2$-equivalent.
\end{corollary}
\begin{proof} The proof goes as in \cite[Thm. 0.1]{PaWhy}, using Lemma \ref{power} and Proposition \ref{free} above.
First, if $G,H$ and $K$ are finite then $G*K$ and $H*K$ are virtually free groups which are not virtually cyclic, since
they are infinite ended by assumption, and hence proper $2$-equivalent by Lemma \ref{power}.\\
\indent Second, if $G$ and $H$ are finite but $K$ is infinite then $G*K$ and $H*K$ contain finite index subgroups
isomorphic to $K* \stackrel{|G|}{\cdots} *K$ and $K* \stackrel{|H|}{\cdots} *K$ respectively, which are proper
$2$-equivalent by Lemma \ref{power}.\\
\indent Similarly, if $G$ and $H$ are infinite but $K$ is finite then $G*K$ and $H*K$ contain finite index subgroups
isomorphic to $G* \stackrel{|K|}{\cdots} *G$ and $H* \stackrel{|K|}{\cdots} *H$, which are proper $2$-equivalent to
$G*G$ and $H*H$ respectively, by Lemma \ref{power}, and hence the conclusion follows as $G*K$ and $H*K$ are then
proper $2$-equivalent to $G*G$ and $H*H$ respectively, which in turn are both proper $2$-equivalent to $G*H$, by
Proposition \ref{free}.\\
\indent Finally, the case when $G,H$ and $K$ are infinite corresponds to Proposition \ref{free}.
\end{proof}
Again, since quasi-isometric (finitely presented) groups are in particular proper $2$-equivalent, \cite[Thm. 0.2]{PaWhy}
together with \cite[Cor. 1.2]{CLQR_AMS} also yields the following
\begin{proposition} \label{amalgam} Let $G$ and $H$ be finitely presented groups and $F$ be a common finite proper
subgroup. If both $G*_F H$ and $G*H$ are infinite ended, then they are proper $2$-equivalent (in fact, they are
quasi-isometric). Similarly, if both $G*_F$ and $G*{\Z}$ are infinite ended, then they are proper $2$-equivalent.
\end{proposition}
\begin{remark} It is worth mentioning that an argument similar to that in the proof of Proposition \ref{free} above
together with the alternative description, within its proper homotopy type, given in the proof of \cite[Thm. 1.1]{CLQR1}
for the corresponding universal cover provide a direct proof of Proposition \ref{amalgam} in the case the factor groups
(resp. the base group) are infinite.
\end{remark}
As a consequence of all of the above, we obtain that the proper $2$-equivalence relation behaves well with respect to
amalgamated products (resp., HNN-extensions) over finite groups; namely,
\begin{theorem} \label{amalgam2} Let $G, G', H$ and $H'$ be finitely presented groups, and $F, F'$ be common finite
proper subgroups of $G, H$ and $G', H'$ respectively. Assume that $G$ is proper $2$-equivalent to $G'$ and $H$ is proper
$2$-equivalent to $H'$. If both $G*_FH$ and $G' *_{F'} H'$ are infinite ended, then they are proper
$2$-equivalent. Similarly, if both $G*_F$ and $G' *_{F'}$ are infinite ended, then they are proper
$2$-equivalent.
\end{theorem}
Observe that any finitely presented group $G$ with more than one end can be decomposed as the fundamental group of a
finite graph of groups $({\mathcal G}, \Gamma)$ whose edge groups are finite and whose vertex groups are finitely
presented groups with at most one end, by Stallings' Structure theorem \cite{Sta} and Dunwoody's accesibility theorem
for finitely presented groups \cite{D}. Thus, as the fundamental group of a graph of groups with $n+1$ edges can be
built out of graphs with fewer edges (by amalgamated products or HNN-extensions), an inductive argument gives us the
following generalisation of Theorem \ref{amalgam2}.
\begin{theorem} \label{graph} Let $G$ and $G'$ be two infinite ended finitely presented groups, and assume they
are expressed as fundamental groups of finite graphs of groups $({\mathcal G}, \Gamma)$ and $({\mathcal G}', \Gamma')$
whose edge groups are finite and whose vertex groups are finitely presented groups with at most one end. If the two
graph of groups decompositions have the same set of proper $2$-equivalence classes of vertex groups (without
multiplicities) then $G$ and $G'$ are proper $2$-equivalent.
\end{theorem}
\begin{remark} In the case $G$ and $G'$ are both semistable at each end, then $({\mathcal G}, \Gamma)$ and $({\mathcal G}',
\Gamma')$ having the same set of proper $2$-equivalence classes of $1$-ended vertex groups amounts to having the same
set of pro-isomorphism types for the fundamental pro-groups of the $1$-ended vertex groups, by Proposition
\ref{characterisation}.
\end{remark}
It is worth mentioning that the converse of Theorem \ref{graph} does not hold in general, unlike the situation under the
quasi-isometry equivalence relation (compare with \cite[Thm. 0.4]{PaWhy}). A counterexample to the converse will be
given in $\S \ref{counterexample}$.

\section{Groups of type $F_n, n \geq 2$}

We recall that a group $G$ has type $F_n$, $n \geq 1$, if there
exists a $K(G,1)$-complex with finite $n$-skeleton. Being of type
$F_1$ is then the same as being finitely generated, and
being of type $F_2$ is the same as being finitely
presented. It is worth mentioning that Bieri-Stallings' groups
provide examples of groups which are of type $F_{n-1}$ but not of
type $F_n$, $n \geq 3$ (see \cite{Bi}). Given $n \geq 2$, we consider
the following relation among groups of type $F_n$, $n \geq 2$.
\begin{definition} \label{pne} Two groups $G$ and $H$ of type $F_n$ are {\it proper $n$-equivalent} if there exist (equivalently, for all) finite $n$-dimensional $(n-1)$-aspherical CW-complexes
$X$ and $Y$, with $\pi_1(X) \cong G$ and $\pi_1(Y) \cong H$, so that their universal covers $\widetilde{X}$ and
$\widetilde{Y}$ are proper $n$-equivalent.
\end{definition}
Again, one can check that any two finite groups are proper $n$-equivalent. It is easy to see that Definition \ref{pne} agrees with Definition \ref{p2e} for $n=2$. Moreover, proper
$(n+1)$-equivalent implies proper $n$-equivalent, $n \geq 2$. As in $\S 2$ (for $n=2$), we next show that this
definition does not depend on the choice of the corresponding CW-complexes, and provide with an alternative equivalent definition via proper homotopy equivalences (after taking wedge
with $n$-spheres).
\begin{theorem} \label{alternative2} Let $G$ and $H$ be two infinite groups of type $F_n$, and let $X$ and $Y$ be any finite $n$-dimensional $(n-1)$-aspherical CW-complexes
with $\pi_1(X) \cong G$ and $\pi_1(Y) \cong H$. Then, if $\widetilde{X}$ and $\widetilde{Y}$ denote the corresponding universal covers, the following statemens are equivalent:
\begin{enumerate}
\item[(a)] The groups $G$ and $H$ are proper $n$-equivalent.
\item[(b)] $\widetilde{X}$ and $\widetilde{Y}$ are proper $n$-equivalent (in the sense of Definition \ref{AQ1}).
\item[(c)] There exist $n$-spherical objects $S^n_\alpha$ and $S^n_\beta$ so that $\widetilde{X} \vee S^n_\alpha$ and $\widetilde{Y} \vee S^n_\beta$ are proper homotopy equivalent.
\item[(d)] The universal covers $\widetilde{X \vee S^n}$ and $\widetilde{Y \vee S^n}$ are proper homotopy equivalent.
\end{enumerate}
\end{theorem}
\noindent Indeed, if $G$ and $H$ are proper $n$-equivalent then there exist finite $n$-dimensional
$(n-1)$-aspherical CW-complexes $W$ and $Z$, with $\pi_1(W) \cong G$ and $\pi_1(Z) \cong H$, so that the universal
covers $\widetilde{W}$ and $\widetilde{Z}$ are proper $n$-equivalent. We now consider $K(G,1)$-complexes $X'$ and $W'$ with $(X')^n = X$ and $(W')^n = W$, and $K(H,1)$-complexes
$Y'$ and $Z'$ with $(Y')^n = Y$ and $(Z')^n = Z$. By the proper cellular approximation theorem, it is not hard to check
that $\widetilde{(X')^n} = \widetilde{X}$ and $\widetilde{(W')^n} = \widetilde{W}$ are proper $n$-equivalent as $X'$ and
$W'$ are homotopy equivalent. It also follows from \cite[Thm. 18.2.11]{Geo}, since $\pi_1(X') \cong \pi_1(W') \cong G$.
Similarly, $\widetilde{Y} = \widetilde{(Y')^n}$ and $\widetilde{Z} = \widetilde{(Z')^n}$ are proper $n$-equivalent. The rest of the argument follows just as in the proof of Theorem \ref{alternative} (see \cite[Thm. 1.1]{CLQR_AMS} and \cite{Zobel}).
\begin{corollary} The relation of being proper $n$-equivalent is an equivalence relation for groups of type $F_n$, $n
\geq 2$.
\end{corollary}
\begin{proof} As in $\S 2$, it follows from the transitivity of proper $n$-equivalences for CW-complexes. Alternatively, one can also show transitivity as follows. Let $G, H$ and $K$ be infinite groups of type $F_n$ so that $G$ is
proper $n$-equivalent to $H$ and $H$ is proper $n$-equivalent to $K$, and let
$X, Y$ and $Z$ be finite $n$-dimensional $(n-1)$-aspherical CW-complexes with $\pi_1(X) \cong G, \pi_1(Y)
\cong H$ and $\pi_1(Z) \cong K$. By Theorem \ref{alternative2}, we have that $\widetilde{X \vee S^n}$ is proper homotopy
equivalent to $\widetilde{Y \vee S^n}$ which in turn is proper homotopy equivalent to $\widetilde{Z \vee S^n}$. Thus,
$G$ and $K$ are proper $n$-equivalent, since $X \vee S^n$ and $Z \vee S^n$ are finite $n$-dimensional $(n-1)$-aspherical
CW-complexes with $\pi_1(X \vee S^n) \cong G$ and $\pi_1(Z \vee S^n) \cong K$.
\end{proof}
\begin{remark}
Again, if $G$ and $H$ are two infinite quasi-isometric groups of type $F_n$ ($n \geq 2$), then it follows from
\cite[Thm. 18.2.11]{Geo} that $G$ and $H$ are proper $n$-equivalent in the above
sense (even if a proper homotopy equivalence is required instead of a proper $n$-equivalence in Definition \ref{pne}, by Theorem \ref{alternative2}). Again, by the \v{S}varc-Milnor Lemma, we have as an immediate consequence that if $H \leq G$ with $[G:H] < \infty$ and $N \leq G$ is a finite normal subgroup then $G$, $H$ and $G/N$ are proper $n$-equivalent to each other.
\end{remark}
Recall that given a group $G$ of type
$F_n$ ($n \geq 1$) and a $K(G,1)$-complex $X$ with finite
$n$-skeleton, we say that $G$ is {\it $(n-1)$-connected at infinity}
if for any compact subset $C \sub \widetilde{X}$ there is a compact
subset $D \sub \widetilde{X}$ so that any map $S^m \lga
\widetilde{X} - D$ extends to a map $B^{m+1} \lga
\widetilde{X} - C$, for all $0 \leq m \leq n-1$ ($0$-connected at
infinity is the same as $1$-ended). It is not hard to check that
$(n-1)$-connectedness at infinity is an invariant under the proper
$n$-equivalence relation, as well as any other proper homotopy
invariant of the group $G$ which depends only up to the
$n$-skeleton of the universal cover of some $K(G,1)$-complex
with finite $n$-skeleton. Likewise, the cohomology group $H^n(G;{\Z}G)$ of a group $G$ of type $F_n$ is an invariant
under proper $n$-equivalences, as it is isomorphic to the cohomology group of the end $H^{n-1}_e(\widetilde{X};{\Z})$,
for any finite $n$-dimensional $(n-1)$-aspherical CW-complex $X$ with $\pi_1(X) \cong G$ (see \cite{Geo}). For a group
$G$ of type $F_n$, the cohomology group $H^n(G;{\Z}G)$ was shown in \cite{Ger} to be a quasi-isometry invariant of the
group.\\

\indent An argument similar to that in Proposition \ref{group-product} yields the following result.
\begin{proposition} \label{group-product2} Let $G, G', H$ and $H'$ be groups of type $F_n$ and assume that $G$ and $H$
are proper $n$-equivalent to $G'$ and $H'$ respectively. Then, $G \times H$ is proper $n$-equivalent to $G' \times H'$.
\end{proposition}
Next, we observe that if $1 \rightarrow N \lga G \lga Q \rightarrow 1$ is a short exact sequence of infinite groups and
$N$ and $Q$ are of type $F_n$, then so is $G$ (see \cite{Geo}). Moreover, it follows from \cite[Prop. 17.3.4]{Geo}
together with Therorem \ref{alternative2} that $G$ is in fact proper $n$-equivalent to $N \times Q$. Thus, using the
same argument as in Proposition \ref{group-extension} we have the following generalisation of Proposition
\ref{group-product2}.
\begin{proposition} \label{group-extension2} Let $1 \rightarrow G \lga K \lga H \rightarrow 1$ and $1 \rightarrow G'
\lga K' \lga H' \rightarrow 1$ be two short exact sequences of groups of type $F_n$, and assume that $G$ and $H$ are
proper $n$-equivalent to $G'$ and $H'$ respectively. Then, $K$ is proper $n$-equivalent to $K'$.
\end{proposition}
Similarly, the same arguments used throughout $\S 3$ yield the following result.
\begin{theorem} \label{graph2} Let $G$ and $G'$ be two infinite ended finitely presented groups, and assume they
are expressed as fundamental groups of finite graphs of groups $({\mathcal G}, \Gamma)$ and $({\mathcal G}', \Gamma')$
whose edge groups are finite and whose vertex groups are of type $F_n$ with at most one end. If the two graph of groups
decompositions have the same set of proper $n$-equivalence classes of vertex groups (without multiplicities) then $G$
and $G'$ are proper $n$-equivalent.
\end{theorem}
Observe that, in the context of Theorem \ref{graph2} above, the groups $G$ and $G'$ are indeed of type $F_n$ (see
\cite[$\S 7.2$, Ex. 3]{Geo}).

\section{The particular case of properly $3$-realizable groups}
\label{App2}

\indent A tower of groups $\underline{P}$ is a {\it telescopic
tower} if it is of the form
\[
\underline{P} = \{ P_0 \stackrel{p_1}{\lgaf} P_1 \stackrel{p_2}{\lgaf} P_2 \lgaf \cdots \}
\]
where $P_i = F(D_i)$ are free groups of basis $D_i$ such that
$D_{i-1} \sub D_i$, the differences $D_i - D_{i-1}$ are finite
(possibly empty), and the bonding homomorphisms $p_k$ are the
obvious projections. An infinite finitely presented group $G$ is
said to be of {\it telescopic type at each end} (resp. {\it at
infinity}, if $G$ is $1$-ended) if its fundamental pro-group at each
end (resp. at infinity) is pro-isomorphic (for any choice of base
ray) to a telescopic tower. Equivalently, $G$ is semistable and pro-free and pro-(finitely generated) at each
end (resp. at infinity), see \cite{L4} for instance. Observe that being of telescopic type at each end is an invariant under proper
$2$-equivalences.\\
\indent It is worth mentioning that
direct products and ascending HNN-extensions of infinite finitely
presented groups have been shown to be of telescopic type at
infinity \cite{CLR,L4}. In fact, any extension of an infinite
finitely presented group
by another infinite finitely presented group is of telescopic type at infinity \cite{CLQRoy}. One-relator groups are
also of telescopic type at each end \cite{CLQR2} (see also \cite{LasRoy}).\\
\indent We recall that a finitely presented group $G$ is {\it
properly $3$-realizable} (abbreviated to P3R) if for some finite $2$-dimensional
CW-complex $X$ with $\pi_1(X) \cong G$, the universal cover
$\widetilde{X}$ of $X$ has the proper homotopy type of a
$3$-manifold (with boundary). Notice that it follows from
\cite{CFLQ} that $\widetilde{X}$ always has the proper homotopy
type of a $4$-manifold. This concept was originally motivated by the Hopf conjecture on the freeness of the second cohomology group $H^2(G; {\Z}G)$ for any finitely presented group $G$; indeed, if $G$ is P3R then one can check that Lefschetz duality yields the freeness of the
cohomology group $H^2(G; {\Z}G)$. Being P3R does not depend on the
choice of $X$ after taking wedge with a single $2$-sphere $S^2$
\cite{ACLQ}. It has been proved that the class of P3R groups is closed under quasi-isometries \cite{CLQR_AMS} and
amalgamated products over
finite groups \cite{CLQR1}, and contains the class of all groups of
telescopic type at each end \cite{L4,LasRoy}. Moreover, it is
conjectured in \cite{FLR} that those two classes are the same (and
examples of non-P3R groups are given), and it has been proved so in \cite{FLR} under the qsf ({{\it quasi-simply filtered}) hypothesis, i.e, roughly speaking, that the corresponding universal cover admits an exhaustion that can be ``approximated" by simply connected complexes (see \cite{BM}). Next, we show that the qsf condition in \cite[Thm. 1.1]{FLR} may as well be
replaced with semistability at infinity in the $1$-ended case, and we count the number of proper $2$-equivalence classes
in this case.
\begin{theorem} \label{p3r} Let $G$ be a $1$-ended finitely presented P3R group. If $G$ is semistable at infinity, then
$G$ is of telescopic type at infinity; moreover, $G$ is proper $2$-equivalent to one (and only one) of the following
groups ${\Z} \times {\Z} \times {\Z}$, ${\Z} \times {\Z}$ or ${\F}_2 \times {\Z}$ (here, ${\F}_2$ is the free group on two generators).
\end{theorem}
Observe that the groups ${\Z} \times {\Z} \times {\Z}$, ${\Z} \times {\Z}$ and ${\F}_2 \times {\Z}$ are $1$-ended and semistable at infinity \cite{M,Mi};
moreover, they are $3$-manifold groups and hence (trivially) P3R groups. Thus, there are exactly three
proper $2$-equivalence classes of $1$-ended and semistable at infinity (finitely presented) groups which contain P3R
groups. It is worth mentioning that each of these proper $2$-equivalence classes contains non $3$-manifold groups as well, see Remark \ref{non3mfld} below.\\

\indent Theorem \ref{p3r} above together with \cite[Thm. 1.2]{CLQRoy}) yield the following
\begin{corollary} If $1 \rightarrow H \lga G \lga Q \rightarrow 1$ is a short exact sequence of infinite finitely presented groups, them $G$ is proper $2$-equivalent to one of the following groups ${\Z} \times {\Z} \times {\Z}$, ${\Z} \times {\Z}$ or ${\F}_2 \times {\Z}$.
\end{corollary}
We also have the following corollary in the infinite ended case.
\begin{corollary} \label{p3r2} Every finitely presented group of telescopic type at each end is the fundamental group of
a finite graph of groups whose edge groups are finite and whose vertex groups are either finite or else proper
$2$-equivalent to one of the following ($1$-ended) groups ${\Z} \times {\Z} \times {\Z}$, ${\Z} \times {\Z}$ or ${\F}_2 \times {\Z}$.
\end{corollary}
\begin{proof} [Proof of Corollary \ref{p3r2}] Observe that any infinite ended finitely presented group $G$ can be
expressed as the fundamental group of a finite graph of groups $({\mathcal G}, \Gamma)$ whose edge groups are finite and
whose vertex groups are finitely presented groups with at most one end, by Stallings' structure theorem \cite{Sta} and
Dunwoody's accesibility theorem \cite{D}. By \cite[Lemma 3.2]{LasRoy}, if $G$ is of telescopic type at each end then
each of the $1$-ended vertex groups in $({\mathcal G}, \Gamma)$ is of telescopic type at infinity as well and hence a
$1$-ended and semistable at infinity P3R group, by \cite[Thm. 1.2]{L4}. The conclusion follows then from Theorem
\ref{p3r}.
\end{proof}

The next subsection is devoted to the proof of Theorem \ref{p3r} above making use of some non-compact $3$-manifold
theory, as well as to the introduction of a new numerical invariant ${\mathfrak P}(G)$ of such group $G$.
\subsection{The fundamental pro-group of $1$-ended and semistable at infinity P3R groups}
\label{App1}

We will use the usual three (co)homologies in proper homotopy theory; namely, ordinary (co)homology ($H_*, H^*$),
end (co)homology ($H^e_*, H_e^*$), cohomology with
compact supports ($H^*_c$) and homology of locally finite chains ($H^{lf}_*$). All coefficients
are in ${\Z}$. Besides the corresponding exact sequences for pairs for each of these
(co)homologies, the three of them  are related by the well-known exact sequences displayed in the
rows of the following diagram for any topological pairs $Z = (Z,Z_0)$ and $Y = (Y,Y_0)$, where
$Z_0 \sub Z$ and $Y_0 \sub Y$ are closed sets.
\begin{center}
\begin{equation}
\label{GD}
\xymatrix@R=15pt{ \ar [r] & H_{q+1}(Z) \ar@{.>}[d]_{\cong}^{D} \ar [r] & H^{lf}_{q+1}(Z)
\ar@{.>}[d]_{\cong}^{D} \ar [r]& H^e_q(Z) \ar@{.>}[d]_{\cong}^{D} \ar [r]& H_q(Z)
\ar@{.>}[d]_{\cong}^{D} \ar [r] &\\
\ar [r] & H_c^{n-q-1}(Y) \ar [r] & H^{n-q-1}(Y) \ar[r]& H_e^{n-q-1}(Y) \ar[r]
& H_c^{n-q}(Y)\ar [r] &\\
}
\end{equation}
\end{center}
Moreover, for any n-manifold $M$ the vertical arrows are duality isomorphisms if we choose either $Z = (M,\partial M)$ and $Y = (M,\emptyset)$
or $ Z= (M , \emptyset)$ and $Y = (M, \partial M)$. We refer to \cite{massey} and \cite{Geo}
for the details. See also \cite{Laitinen} for a similar approach.\\

\indent Given an arbitrary P3R group $G$, let ${\mathcal M}_G$ be the set of $3$-manifolds having the proper homotopy type
of the universal cover $\widetilde{X}$ for some finite $2$-dimensional CW-complex $X$ with $\pi_1(X) = G$. Any manifold $M \in {\mathcal M}_G$
will be said to be an {\it associated $3$-manifold} to the P3R group $G$.
In this subsection all P3R groups are $1$-ended so that $\mathcal M_G$ consists of $1$-ended simply connected $3$-manifolds with boundary.
Let us start with some properties derived from duality of those manifolds.
\begin{proposition}\label{p1} Let $M$ be a $1$-ended simply connected $3$-manifold with boundary. Then, the boundary $\partial M$
consists of a (locally finite) union of planes and spheres. In particular, $H^{lf}_1(\partial M) =0$. Moreover, $\pi_2(M) \cong H_2(\partial M)$.
\end{proposition}
\begin{proof} The $H^*_c$-exact sequence of the pair $(M,\partial M)$ and duality show that
\[
0 = \limdi H^1(M,U_j) \cong H^1_c(M) \lga H^1_c(\partial M)
\lga H^2_c(M,\partial M) \cong H_1(M)=0
\]
\noindent is exact, whence $H^1_c(\partial M)$ vanishes. Therefore, $H^1_c(W)=0$ for each component $W \sub \partial M$, and the
classification of the open surfaces in \cite{R} shows that $W$ is either $S^2$ or ${\R}^2$, and hence
$H^{lf}_1(\partial M) = 0$. Here, $\{U_j\}$ is any system of $\infty$-neighborhoods in $M$, see \cite{Geo} or \cite{massey}.\\
\indent Since $G$ is $1$-ended, we have $0=H^1_c(M)$ which, by duality, implies that $H_2(M,\partial M)=0$. Thus, the inclusion $\partial M \subseteq M$ induces isomorphisms $H_2(\partial M)\cong H_2(M)\cong \pi_2(M)$.
\end{proof}
According to Proposition \ref{p1}, if $X$ is some finite $2$-dimensional CW-complex with $\pi_1(X)\cong G$ so that $\widetilde{X}$ has the proper homotopy type of $M \in \mathcal M_G$, then $\pi_2(X)$ is carried by $\partial M$. Furthermore, if $G$ is infinite, then $\pi_2(M)$ is either infinitely generated or trivial. Indeed, any non-trivial $2$-cycle would be translated by the $G$-action on $\widetilde{X}$ to infinitely many other non-trivial $2$-cycles in $H_2(\widetilde{X})$. Thus, if $\pi_2(M)(\cong\pi_2(X))$ is not trivial, then it must be infinitely generated, yielding then the following result.
\begin{proposition} \label{infinitepi2}
If $M$ is a $1$-ended simply connected $3$-manifold with boundary and $\pi_2(M)\neq 0$, then $\pi_2(M)$ is an infinitely generated free abelian group.
\end{proposition}
The cardinal number ${\mathfrak p}(M) = \# \mbox{planes in } \partial M$ can be easily estimated by the end (co)homology groups. In fact, the existence of planes in $\partial M$ of a $1$-ended simply connected $3$-manifold $M$ is determined by the following proposition.
\begin{proposition}\label{p2} Let $M$ be a $1$-ended simply connected $3$-manifold. The following statements are equivalent:
\begin{enumerate}
\item[(a)] $\partial M$ contains at least one plane.
\item[(b)] $H^e_0(\partial M) \stackrel{duality}{\cong} H^1_e(\partial M) \neq 0$ (moreover, ${\mathfrak
			p}(M) = rank H^e_0(\partial M)$).
\item[(c)]  The homomorphism $j_*: {\Z} \cong H^e_0(M) \lga H_0^e(M,\partial M)$ is trivial.
\end{enumerate}
\end{proposition}
\begin{proof} The Milnor exact sequence
\[
0 \lga {\liminv}^1 pro-H_1(\partial M) \lga H_0^e(\partial M) \lga \liminv pro-H_0(\partial M) \lga 0
\]
\noindent and Proposition \ref{p1} yield an isomorphism
\[
\displaystyle H^e_0(\partial M) \cong \liminv pro-H_0(\partial M) \cong \bigoplus_{i= 1} ^{m} {\Z}
\]
\noindent where $m$ is number of planes in $\partial M$ (possibly $m = \infty$). This shows (a) $\Longleftrightarrow$ (b) as well as the equality
${\mathfrak p}(M) = rank \ H^e_0(\partial M)$.\\
\indent On the other hand, the exact sequence corresponding to the upper row in (\ref{GD}) together with
duality and \cite[Prop 11.1.3]{Geo} and \cite[Prop 12.1.2]{Geo} yield the exact sequence
\[
0 = H_1^{lf}(M) \lga H_0^e(M) \lga H_0(M) \cong {\Z} \lga H_0^{lf}(M) = 0
\]
\noindent whence $H^e_0 (M) \cong {\Z}$, and hence the exact $H^e_*$-sequence of $(M,\partial M)$ leads to the exact sequence
\[
H^e_0(\partial M) \stackrel{i_*}{\lga} H^e_0 (M)  \cong {\Z} \stackrel{j_*}{\lga} H^e_0(M,\partial M) \lga H^e_{-1}(\partial M) \lga 0 = H^e_{-1}(M)
\]
Therefore, if $\partial M$ contains at least one plane, then the end of $M$ can be reached by a ray in
$\partial M$ and so $i_*$ above is onto, whence $j_*$ is trivial. Conversely, the annihilation of
$j_*$ implies that $H^e_0(\partial M) \neq 0$ and hence $\partial M$ contains
at least one plane as proved above. This shows (a) $\Longleftrightarrow$ (c).
\end{proof}
As a consequence of Proposition \ref{p2}, we get the following sufficient condition for the
existence of planes in $\partial M$.
\begin{proposition}\label{p3} The end (co)homology groups $H^1_e(M) \cong H^e_1(M,\partial M)$ are isomorphic free abelian groups. Moreover, if ${\mathfrak p}(M) \neq 0$ then ${\mathfrak p}(M) = 1+ rank \ H^1_e(M)$, and ${\mathfrak p}(M) \geq 2$ if and only if $ H^1_e(M)\neq 0$.
\end{proposition}
\begin{proof} Recall the duality isomorphism $H^1_e(M) \cong H^e_1(M,\partial M)$ from (\ref{GD}). From the exactness of the sequences
\[
0=H_1(M) \lga H_1(M,\partial M) \lga H_0(\partial M)
\]
\[
0 \stackrel{(1)}{=} H^{lf}_{2}(M,\partial M) \lga H^{e}_{1}(M, \partial M) \lga H_1(M,\partial M) \lga \cdots
\]
where (1) follows from duality (see \ref{GD}), we obtain that $H_1^e(M,\partial M)$ is isomorphic to a subgroup of the free abelian group $H_0(\partial M)$. Moreover, we have a commutative diagram
\[
\xymatrix@R=10pt@C=10pt{
& H^{lf}_{2}(M)\ar[d]^{(2)}\ar[r] & H^{lf}_{2}(M,\partial M) \stackrel{(1)}{=}0 \ar[d] & & &&\\
& H^{e}_1(M) \ar [r]^{k_*} & H^{e}_1(M,\partial M) \ar [r] & H^{e}_0(\partial M) \ar[r ]& H^{e}_0(M) \stackrel{(3)}{\cong} {\Z} \ar [r]^{j_*} & H^e_0(M,\partial M)
}
\]
\noindent where (3) was observed in the proof of Proposition \ref{p2}, and the homomorphism (2) is onto as it is followed by the zero homomorphism $H_1^e(M) \lga H_1(M)=0$, and so $k_*$
is the trivial homomorphism.\\
\indent Moreover, if $\partial M$ contains planes then $j_* = 0$ by Proposition \ref{p2} and the exactness of the bottom sequence yields
$H^e_0(\partial M) \cong H^e_1(M,\partial M) \oplus {\Z}$. Therefore,
\[
{\mathfrak p}(M) = rank \ H^e_0(\partial M) = rank \ H^e_1(M,\partial M) + 1 = rank \ H^1_e(M) +1
\]
Here we use again Proposition \ref{p2}. If $H^1_e(M) \cong H^e_1(M,\partial M) \neq 0$ then
$H^e_0(\partial  M) \neq 0$; i.e, $\partial M$ contains planes and, by the equality above, ${\mathfrak p}(M) \geq 2$.
\end{proof}
As a corollary, we have the following result.
\begin{corollary} \label{c1} If $M$ and $N$ are two proper homotopy equivalent $1$-ended simply connected $3$-manifolds and ${\mathfrak p}(N)\geq 2$, then ${\mathfrak
		p}(M) = {\mathfrak p}(N)$.
\end{corollary}
\begin{remark} \label{r0}
\begin{enumerate}[(a)]
\item Notice that the vanishing of $H^1_e(M) \cong H^e_1(M,\partial M)$ only implies
${\mathfrak p}(M) \leq 1$ but not necessarily the absence of planes in $\partial M$.
\item Corollary \ref{c1} may fail for $\mathfrak p (N) \leq 1$. Indeed, by regarding ${\R}^3_+$ as the infinite cylinder $X = B^2 \times {\R}_+$, consider a sequence $\{D_n\}_{n \geq 0}$ of $3$-balls with $D_n \subset int (B^2 \times [n,n+1])$. Then $\displaystyle N = {\R}^3_+ - (\bigcup_{n=0}^\infty D_n)$ admits the subspace $\displaystyle X^2 = (S^1 \times {\R}_+) \cup (\bigcup_{n= 0}^\infty B^2 \times \{n\})$ as a proper strong deformation retract. Notice that $X^2$ has the proper homotopy type of the spherical object $S^2_{{\R}_+}$ obtained by attaching one $2$-sphere at each vertex $t \in {\N} \cup \{0\} \subseteq {\R}_+$. Similarly, we write ${\R}^3 = B^3 \cup S^2 \times [1,\infty)$, where $B^3$ is the closed unit ball. Moreover, we consider $S^2$ as the attaching of a $2$-cell at a point $\{x_0\}$, so that each cylinder $S^2 \times [n,n+1]$ turns out to be the attaching of a $3$-cell $D_n$ at $(\{x_0\} \times [n,n+1]) \cup (S^2 \times \{n,n+1\})$. By choosing, for each $n \geq 1$, a $3$-ball $E_n \subset int(D_n)$ and $E_0 \subset int(B^3)$, it is clear that $\displaystyle M = {\R}^3- (\bigcup_{n=0}^\infty E_n)$ properly deforms onto
$\displaystyle (\{x_0\} \times [1,\infty)) \cup (\bigcup_{n=1}^\infty S^2 \times \{n\})$, which is proper homotopy equivalent to the spherical object $S^2_{{\R}_+}$ above. Hence, $M$ and $N$ have the same proper homotopy type but $\mathfrak p(M)=0$ and $\mathfrak p(N)=1$. Observe that this construction can be carried out in any dimension and for any (locally finite) countably infinite collection of top-dimensional balls.
\end{enumerate}
\end{remark}
Given any $1$-ended P3R group $G$ (we do not assume semistability) and two $2$-dimensional CW-complexes $X$ and $Y$ with $\pi_1(X)\cong G \cong \pi_1(Y)$, the universal covers $\widetilde{X \vee S^2}$ and $\widetilde{Y \vee S^2}$ are proper homotopy equivalent by \cite[Thm. 1.1]{CLQR_AMS}. Moreover, by the proof of \cite[Prop. 1.3]{ACLQ}, if $\widetilde{X}$ has the proper homotopy type of a $3$-manifold $M$ then $\widetilde{X \vee S^2}$ has the proper homotopy type of a $3$-manifold $M'$ with $\mathfrak p(M)= \mathfrak p(M')$. Thus, if $\mathfrak p(M) \geq 2$ then for any other $3$-manifold $N \in \mathcal M_G$ we have $\mathfrak p(M)=\mathfrak p(M')=\mathfrak p(N')=\mathfrak p(N)$, by Corollary \ref{c1}. However, as Remark \ref{r0}(b) points out, this is not true if $\mathfrak p(M)\leq 1$. Notice that the complement of a discrete countable family of open $3$-balls in ${\R}^3$ has the proper homotopy type of the universal cover of the $2$-skeleton of an irreducible closed $3$-manifold.\\
\indent Notwithstanding, if we set
\begin{equation} \label{dp}
\mathfrak P(G)=min\{\mathfrak p(M); M \in \mathcal M_G\}
\end{equation}
\noindent we get a numerical invariant of the $1$-ended P3R group $G$. This number will be called the $\partial$-number of $G$. In case ${\mathfrak P}(G) \neq 0$, this number is determined by the cohomology of $G$ as follows (cf. \cite{L0}).
\begin{proposition} \label{p4} The second cohomology $H^2(G;{\Z}G)$ of any $1$-ended P3R group $G$ is a free abelian group. Moreover, if ${\mathfrak P}(G) \neq 0$ then
\begin{equation} \label{nplanos}
{\mathfrak P}(G) = \mbox{rank } \ H^2(G;{\Z}G) +1
\end{equation}
\end{proposition}
\begin{proof} Let $X$ be a finite $2$-dimensional CW-complex with $\pi_1(X)\cong G$ and so that $\widetilde{X}$ has the proper homotopy type of some $M \in {\mathcal M}_G$. We have isomorphisms
$H^1_e(M) \cong H^{1}_e(\widetilde{X}) \cong H^2(G;{\Z}G)$, according to \cite[Cor 13.2.9]{Geo} and \cite[Cor. 13.2.13]{Geo}. Moreover, being $G$ P3R the cohomology group $H^2(G;{\Z}G)$ is free abelian (see \cite{L0}). Therefore,
\[
{\mathfrak P}(G) = \mbox{rank } H^1_e(M) + 1 = \mbox{rank } H^2(G;{\Z}G) + 1
\]
\noindent by Proposition \ref{p3}.
\end{proof}
\begin{corollary} \label{c2} If the $1$-ended P3R group $G$ contains at least one element of infinite order (in particular, if $G$ is torsion-free) then the possible values for
${\mathfrak P}(G)$ are $0,1,2$ or $\infty$.
\end{corollary}
\begin{proof} This is a direct consequence of equality (\ref{nplanos}) and Corollary 5.2 in \cite{Far}, stating that the rank of $H^2(G;{\Z}G)$ is $0$,$1$ or $\infty$.
\end{proof}
\begin{remark} \label{rnplanos} Notice that equality (\ref{nplanos}) holds for any commutative ring $R$, i.e.,
\begin{equation}\label{nplanosR}
{\mathfrak P}(G) = \mbox{rank } \ H^2(G;RG) +1
\end{equation}
Indeed, since $M \in {\mathcal M}_G$ is orientable and \cite[Cor. 13.2.9]{Geo} and \cite[Cor. 13.2.13]{Geo} hold for an
arbitrary commutative ring $R$, we have isomorphisms $H^1_e(M;R) \cong H^{1}_e(\widetilde{X};R)\cong H^2(G;RG)$ and the same proof as in Proposition \ref{p3} yields
the equality (\ref{nplanosR}).
\end{remark}
\begin{proposition} \label{2dualitygroup} If $G$ is a $1$-ended P3R group with $cd(G)\leq 2$, then ${\mathfrak P}(G)\geq 2$.
\end{proposition}
\begin{proof} Since $G$ is $1$-ended and $cd(G)\leq 2$, then $G$ is a $2$-dimensional duality group by \cite[Thm. 5.2]{BieriEckmannDualityGroups} (see also the argument in \cite[Prop. 9.17(b)]{Bi}). Thus, $H^2(G;{\Z}G) \neq 0$ and we get ${\mathfrak P}(G)\geq 2$, by Proposition \ref{p4}.
\end{proof}
\begin{corollary} \label{gemdim2} If $G$ is a $1$-ended P3R group with $geom \; dim (G)= 2$, then ${\mathfrak P}(G) \geq 2$.
\end{corollary}
\begin{remark} \label{finitepi2}
Notice that in case $\pi_2(M)=0$ for some $3$-manifold $M \in {\mathcal M}_G$, then $geom\ dim (G)= 2$ and hence ${\mathfrak P}(G) \geq 2$.
\end{remark}
For $\mathfrak{P}(G)=2$ we have a stronger converse to Corollary \ref{gemdim2} stated as Theorem \ref{vsg} below. Recall that a group $G$ satisfies {\it virtually} the property $\mathcal{P}$ if $G$ contains a subgroup $H \leq G$ of finite index satisfying $\mathcal{P}$. The property $\mathcal{P}$ is said to be {\it virtual} if
it holds for any group which satisfies $\mathcal{P}$ virtually. For instance, the number of ends and
semistability are well-known virtual properties of a group. In \cite{ACLQ} it was shown that proper $3$-realizability is a virtual property. Next, we will enhance \cite[Thm. 1.1]{ACLQ} by proving that the $\partial$-number ${\mathfrak P}(G)$ is a virtual property of the P3R group $G$. Namely,
\begin{theorem} \label{vnp} Let $G$ be a $1$-ended finitely presented group $G$, and let $H \leq G$ be a subgroup of finite index. Then $G$
is P3R if and only if so is $H$, and in this case ${\mathfrak P}(G) = {\mathfrak P}(H)$.
\end{theorem}
\begin{proof} By \cite[Thm. 1.1]{ACLQ} we already know that the proper $3$-realizability of $G$ is equivalent to that of $H$. Moreover, the proof of \cite[Thm. 1.1]{ACLQ} shows that for any finite $2$-dimensional CW-complexes $X$ and $Y$ having $G$ and $H$ as fundamental groups respectively, there are finite bouquets of $2$-spheres $\displaystyle \bigvee_{i=1}^{m} S^2_i$ and $\displaystyle \bigvee_{j=1}^{n} S^2_j$ so that the universal covers of $W= X \vee (\displaystyle \bigvee_{i=1}^{m} S^2_i)$ and $Z= Y \vee (\displaystyle \bigvee_{j=1}^{n} S^2_j)$ are proper homotopy equivalent. Observe that $\widetilde{W·}$ is proper homotopy equivalent to $\widetilde{X \vee S^2}$, and similarly $\widetilde{Z}$ is proper homotopy equivalent to $\widetilde{Y \vee S^2}$ (see the proof of \cite[Prop. 1.3]{ACLQ}). Thus, any $3$-manifold $M \in \mathcal M_G$ with the same proper homotopy type of $\widetilde{X}$ (or of $\widetilde{X \vee S^2}$) yields a $3$-manifold $M' \in \mathcal M_H$ with the same proper homotopy type of $\widetilde{Y \vee S^2}$ and with $\mathfrak p(M) = \mathfrak p(N)$, and viceversa. Therefore, $\mathfrak P(G)=\mathfrak P(H)$.
\end{proof}
Next, we point out that the $1$-ended P3R groups with $\partial$-number $2$ are exactly the virtually surface groups.
\begin{theorem} \label{vsg} A $1$-ended finitely presented P3R group $G$ has ${\mathfrak P}(G) = 2$ if and only if $G$ is a virtually surface group.
\end{theorem}
\begin{proof} Assume $G$ is a $1$-ended P3R group with ${\mathfrak P}(G) = 2$. Then, $H^2(G;\mathbb{F}G) \cong \mathbb{F}$
for any field $\mathbb{F}$ by Remark \ref{rnplanos}, and \cite[Thm. 0.1]{B} yields that $G$ is a virtually surface group. Here we use the fact that any finitely presented group is an $FP_2$ group
for any commutative ring $R$. The converse is an immediate consequence of Theorem \ref{vnp}.
\end{proof}	
\indent The previous results do not assume the semistability at infinity of the $1$-ended P3R group $G$. Under this aditional hypothesis, the $3$-manifolds in $\mathcal M_G$ are then semistable at infinity. Thus, Perelman's results on the Poincar\'e Conjecture \cite{Morgan} lead us to the following theorem.
\begin{theorem} \label{t1.1} Let $M$ be a $1$-ended simply connected and semistable $3$-manifold, then $M$ has the proper homotopy
type of the complement in ${\R}^3$ of a disjoint union of locally finite families of pairwise disjoint open $3$-balls and open half-spaces.
\end{theorem}
\begin{proof} We already know that $\partial M$ is a union of planes and spheres, by Proposition \ref{p1}. Let $\widehat{M}$ be the manifold obtained by attaching copies of the $3$-ball and the
$3$-dimensional half-space thus capping off the boundary of $M$, and let $\widetilde{M} \sub \widehat{M}$ be the submanifold obtained by attaching only the corresponding copies of the $3$-ball. Since $M$ does not contains fake $3$-balls (by the solution to the Poincar\'e Conjecture \cite{Morgan}) it follows that $\widehat{M}$ is an irreducible semistable and contractible open $3$-manifold and hence homeomorphic to ${\R}^3$, see \cite{K} or \cite[(A) on p. 213]{BT}. Similarly, $\widetilde{M}$ is an irreducible semistable and contractible $3$-manifold with $\partial \widetilde{M} \neq \emptyset$, and hence $\widetilde{M}$ is a missing boundary manifold of the form $B^3-Z$, where $Z \sub \partial B^3 = S^2$ is a closed subset with $S^2-Z$ being a disjoint union of open disks in one-to-one correspondence with the planes in $\partial M$. See \cite[Prop. 8.2]{BT2}.
\end{proof}
\begin{corollary} \label{t1} Let $G$ be a $1$-ended and semistable at infinity P3R group. Any associated $3$-manifold $M \in \mathcal M_G$ has the proper homotopy
type of the complement in ${\R}^3$ of a disjoint union of locally finite families of pairwise disjoint open $3$-balls and open half-spaces.
\end{corollary}
\begin{remark} \label{values}
\begin{enumerate}[(a)]
\item For a $1$-ended and semistable at infinity P3R group $G$, the possible value $\mathfrak{P}(G)=1$ from Corollary \ref{c2} cannot occur. Indeed, the only possible $3$-manifold  $M \in \mathcal M_G$  with $\mathfrak{p}(M)=1$ would the half-space ${\R}^3_{+}$, by Remark \ref{r0}(b) and Corollary \ref{t1}, which is ruled out by Remark \ref{finitepi2}.
\item For any $1$-ended and semistable at infinity P3R group $G$, the set $\mathcal M_G$ contains at most two different proper homotopy types determined by the $\partial$-number $\mathfrak{P}(G)$ and the presence of $2$-spheres in the boundary of the $3$-manifolds (see the proof of \cite[Prop. 1.3]{ACLQ}). If $\mathfrak{P}(G) \neq 0$ then, by Corollary \ref{t1} and the description in \cite[Thm. I']{BT2}, we could argue that those proper homotopy types are also the topological ones. For $\mathfrak{P}(G)=0$ we have one proper homotopy type and two topological types, by Remark \ref{r0}. On the other hand, in the infinite ended case the set $\mathcal M_G$ may contain infinitely many different topological types (see the example given in the Appendix).
\end{enumerate}
\end{remark}
Corollary \ref{t1} allows us to determine easily the fundamental pro-group of any $3$-manifold $M\in\mathcal M_G$ associated to any $1$-ended and semistable at infinity P3R group $G$. For this, we start by choosing an appropriate system of $\infty$-neighborhoods of $M$ in terms of the families  $\{\Pi_m\}_{m \geq 1}$ and $\{\Sigma_k\}_{k\geq 1}$ (possible empty, but not simultaneously since $M$ is proper homotopy equivalent to a $2$-dimensional CW-complex, see \cite[Prop.3.1]{ACLQ}) of planes and spheres in $\partial M$ respectively.
\begin{lemma} \label{lemaentornos} Under the above hypotheses, there is a system of $\infty$-neighborhoods of $M$ consisting of connected $3$-manifolds with
boundary $\{U_j\}_{ \geq 0}$ with $U_0 = M$ and such that the following two conditions hold:
\begin{enumerate}[(a)]
\item The spheres $\Sigma_k$ miss all the topological frontiers $Fr(U_j)$.
\item For each $j \geq 0$, there is an integer $m(j) \geq 0$ such that the plane $\Pi_m$ is contained in $U_j - Fr(U_j)$ if $m \geq m(j)$, and $Fr(U_j)$ is a compact surface with boundary
$\displaystyle \bigcup_{m=1}^{m(j)} \Pi_m \cap Fr(U_j)$; moreover, there exist homeomorphisms of pairs
$\displaystyle \left( \Pi_m \cap U_j, \Pi_m \cap Fr(U_j) \right) \cong \left( S^1 \times [m(j)-m, \infty), S^1 \times \{m(j)-m\} \right)$.
\end{enumerate}
\end{lemma}
\begin{proof} The $U_j$'s are found as follows. We use \cite[Lemma 3.1]{Brown} to start with a system ${\mathcal W} = \{W_j\}_{j \geq 0}$ of
$\infty$-neighborhoods of $M$, with $W_0 = M$ and each $W_j$ being
a connected $3$-submanifold whose topological frontier $Fr(W_j)$ is a connected compact
surface, possibly with boundary. By using a regular neighborhood of the union of each $W_i$ with
the (finitely many) spheres hitting it and disjoint with the rest of the spheres, we replace
${\mathcal W}$ by a new system ${\mathcal W}' = \{W'_j\}_{j \geq 0}$ of $\infty$-neighborhoods of $M$,
already satisfying condition (a). Then, one observes that the intersections $\{W'_j\cap \Pi_m\}_{j\geq 0}$ form a system of $\infty$-neighborhoods of
the plane $\Pi_m$ and we can assume that each intersection $W'_j \cap \Pi_m$ distinct from $\Pi_m$ is contained in
an infinite cylinder $C'_{j,m}$ such that $\{C'_{j,m}\}_{j \geq 0}$ is a system of $\infty$-neighborhoods in $\Pi_m$.
This way, the polyhedra $W''_j = W'_j \cup_{m\geq 1} C'_{j,m}$ give a new system of
$\infty$-neighborhoods in $M$ and by using a regular neighborhood of each $W''_j$ avoiding the spheres outside it, we
can replace each $W''_j$ by a $3$-submanifold $U_j$ whose intersection with each plane $\Pi_m$ is either the whole $\Pi_m$ or a new infinite
cylinder $C_{j,m}$ with $Fr(U_j) \cap C_{j,m} = \partial C_{j,m}$. Finally, $\{U_j\}_{j \geq 0}$ satisfies (a) and (b).
\end{proof}
\begin{theorem} \label{t1bis} Let $G$ be a $1$-ended and semistable at infinity P3R group. Then, for any $3$-manifold $M \in \mathcal M_G$ associated to $G$ (and any choice of base ray) the tower $pro-\pi_1(M)$ is a telescopic tower generated by a set of $\mathfrak{P}(G)-1$ elements if $\mathfrak{P}(G)\neq 0$ and the trivial tower otherwise.
\end{theorem}
\begin{proof} Let $M \in \mathcal M_G$. First, observe that the towers $pro-\pi_1(M) \cong pro-\pi_1(\widetilde{M})$ are pro-isomorphic, where $\widetilde{M}$ is the $3$-manifold obtained from $M$ by attaching copies of the $3$-ball as in Theorem \ref{t1.1} and so that $\partial \widetilde{M}$ contains no spheres. Indeed, by using a system $\{U_j\}_{j \geq 0}$ of $\infty$-neighborhoods for $M$ as in Lemma \ref{lemaentornos} and the
corresponding system $\{\widetilde{U}_j\}_{j \geq 0}$ of $\infty$-neighborhoods for $\widetilde{M}$, one can readily check that the inclusion $i : M \subseteq \widetilde{M}$ induces isomorphisms $i_* : \pi_1(U_j) \lga \pi_1(\widetilde{U}_j), j \geq 0$.\\
On the other hand, the pair $(\widetilde{M}, \partial \widetilde{M})$ is homeomorphic to $(B^3-Z, S^2-Z)$, with $\displaystyle S^2-Z = \cup_i D_i$ being a pairwise disjoint family of open disks $D_i$ of cardinality $\mathfrak{P}(G)$.\\
\indent If $\mathfrak{P}(G) = 0$ then $\widetilde{M}\cong {\R}^3$ is simply connected at infinity, and the case $\mathfrak{P}(G) = 2$ follows easily from Theorem \ref{vsg}, since $pro-\pi_1(M)$ is then pro-isomorphic to the constant tower $1 \leftarrow {\Z} \stackrel{id}{\longleftarrow} {\Z} \stackrel{id}{\longleftarrow} \cdots $. Otherwise (i.e., $\mathfrak{P}(G) = \infty$ by Remark \ref{values}(a)), we fix a homeomorphism $h_i: int \ B^2 \lga D_i$ for each $i$, and denote by  $D_{i,k}$  the image by $h_i$  of the ball $B^2_k\subset B^2$ of radius $k/k+1$  for $k\geq 1$.  Similarly, let  $B^3_k\subset B^3$  be  the  $3$-ball of the same radius.
\par
Let  $C_{i,k}$  denote the cone over  $D_{i,k}$  with  vertex  the center of $B^3$.  By identifying  the compact sets  $\displaystyle N_n = B^3_n \cup (\bigcup_{s= 1}^{n} \{C_{s,t}; s+t = n+1\}) \subset B^3-Z$ with their  images by the homeomorphism $(\widetilde{M}, \partial \widetilde{M}) \cong  (B^3-Z, S^2-Z)$,  the following facts can be readily checked  :
\begin{enumerate}[(i)]
\item $\displaystyle \widetilde{M} = \bigcup_{n \geq 1} N_n$ and $\{\widetilde{M} - N_n\}_{n \geq 1}$ is a system of $\infty$-neighborhoods for $\widetilde{M}$.
\item $\widetilde{M} - N_n$ is homotopy equivalent to a $2$-sphere with $n$ holes,  whence $\pi_1(\widetilde{M} - N_n)$ is a finitely generated free group of rank $n-1$.
\item  Via the previous homotopy  equivalence,  the  homomorphism  between fundamental  groups induced  by the inclusion  $\widetilde{M} - N_{n+1} \sub \widetilde{M} - N_n$  anihilates the generator arising when removing  $C_{n+1,1}$  from  $\widetilde{M} - N_n$.
\end{enumerate}
This way we have proved that
\[
pro-\pi_1(M) \cong pro-\pi_1(\widetilde{M}) \equiv \{1 \longleftarrow \pi_1(\widetilde{M} - N_1) \longleftarrow \pi_1(\widetilde{M} - N_2) \longleftarrow \cdots\}
\]
is a telescopic tower as claimed in the theorem.
\end{proof}

We now finish this subsection with the proof of Theorem \ref{p3r} as a consequence of all of the above.
\begin{proof}[Proof of Theorem \ref{p3r}] Let $G$ be a $1$-ended finitely presented P3R group which is semistable at
infinity, and let $X$ be a finite $2$-dimensional CW-complex with $\pi_1(X) \cong G$ an whose universal cover is proper
homotopy equivalent to a $3$-manifold $M$. By Theorem \ref{t1bis}, we have that $G$ is indeed of telescopic type at
infinity. Furthermore, its fundamental pro-group at infinity is pro-isomorphic to a telescopic tower $1 \leftarrow P_1
\lgaf P_2 \lgaf \cdots$ of one of the following three types:
\begin{enumerate}
\item[(i)] $P_i = \{1\}$, for all $i \geq 1$ (if $\mathfrak{P}(G)=0$).
\item[(ii)] $P_i={\Z}$, for all $i \geq 1$ (if $\mathfrak{P}(G)=2$).
\item[(iii)] $P_i=F(D_i)$ with $D_i \subsetneq D_{i+1}$ and $D_{i+1} - D_i$ finite, for all $i \geq 1$ (if
$\mathfrak{P}(G)= \infty$).
\end{enumerate}
We recall that two towers (inverse sequences) of groups $G_0 \stackrel{\lambda_1}{\lgaf} G_1 \stackrel{\lambda_2}{\lgaf}
G_2 \lgaf \cdots$ and $H_0 \stackrel{\mu_1}{\lgaf} H_1 \stackrel{\mu_2}{\lgaf} H_2 \lgaf \cdots$ are pro-isomorphic if
after passing to subsequences there exists a commutative diagram:
\[
\xymatrix{
{G_{i_0}} & {G_{i_1}} \ar[l] \ar[dl] & {G_{i_2}} \ar[l] \ar[dl] & \cdots \ar[l] \ar[dl]\\
{H_{j_0}} \ar[u] & {H_{j_1}} \ar[l] \ar[u] & {H_{j_2}} \ar[l] \ar[u] & \cdots \ar[l]\\
}
\]
where the horizontal arrows are the obvious compositions of the corresponding $\lambda$'s or $\mu$'s.\\
\indent Thus, one can easily check that all telescopic towers as in $(iii)$ are pro-isomorphic to each other. Moreover,
it is not hard to check that the telescopic towers of types $(i), (ii)$ and $(iii)$ above can be realized as the
fundamental pro-group of the following $1$-ended finitely presented groups, respectively:
\begin{enumerate}
\item[(i)] The direct product $G={\Z} \times {Z} \times {\Z}$, with the finite $2$-dimensional CW-complex $X$ being the
$2$-skeleton of $S^1 \times S^1 \times S^1$.
\item[(ii)] The direct product $G={\Z} \times {\Z}$, with the finite $2$-dimensional CW-complex $X=S^1 \times S^1$.
\item[(iii)] The direct product ${\F}_2 \times {\Z}$, with the finite $2$-dimensional CW-complex $X=(S^1 \vee S^1) \times S^1$.
\end{enumerate}
The second part of Theorem \ref{p3r} follows then from Proposition \ref{characterisation} since the (pro-isomorphism
type of) the fundamental pro-group of $G$ completely determines its proper $2$-equivalence class.
\end{proof}
\begin{remark} \label{non3mfld} Notice that each of the three proper $2$-equivalence classes above contains non $3$-manifold groups. For this, we may consider the following $1$-ended and semistable at infinity finitely presented groups. First, let $G$ be right angled Artin group associated to the flag complex given by a $3$-simplex (with its $1$-skeleton as defining graph). Then, by \cite[Prop. 5.7(iii)]{HermillerMeier} $G$ is a non $3$-manifold group which is simply connected at infinity, by \cite[Cor. 5.2]{BradyMeier}, and hence proper $2$-equivalent to ${\Z} \times {\Z} \times {\Z}$. Second, the finitely generated abelian group $H = {\Z} \times {\Z} \times {\Z}_2$ is a non $3$-manifold group (see \cite[Thm. 9.13]{Hempel}) which has ${\Z} \times {\Z}$ as a subgroup of finite index and hence it is proper $2$-equivalent to ${\Z} \times {\Z}$. And finally, let $K$ be the Baumslag-Solitar group $\langle a, t; t^{-1}at=a^2 \rangle$ which is a non $3$-manifold group (see \cite{Jaco}) and has a fundamental pro-group of telescopic type as in (iii) above (see \cite{L4} for details) and hence $K$ is proper $2$-equivalent to ${\F}_2 \times {\Z}$.
\end{remark}
\begin{remark} \label{GG}
\begin{enumerate}[(a)]
\item Given a $1$-ended finitely presented group $G$ containing an element of infinite order, we may interpret Theorems \ref{vsg} and \ref{t1bis} above as a converse to \cite[Thm. 1.4]{GeoG} in the following fashion:
	\begin{enumerate}[(i)]
		 \item The group $G$ has pro-monomorphic fundamental group at infinity (actually pro-stable) if and only if $G$ is a semistable $P3R$ group with $\mathfrak{P}(G)=0$ or $\mathfrak{P}(G)=2$. As pointed out in \cite[Remark 6]{GeoG} or \cite[Prop. 2.8]{GeoG}, we may talk of the fundamental group at infinity of a group $G$ since the conclusion is that the group $G$ is semistable.\\
\indent Also, Theorems \ref{t1bis} and \ref{alternative} and Proposition \ref{characterisation} yield the following:
\item The group $G$ has as fundamental group at infinity a strictly telescopic tower (meaning pro-epimorphic not pro-stable) if and only if $G$ is a semistable $P3R$ group with $\mathfrak{P}(G)=\infty$.
\end{enumerate}
\item Notice that in both cases the $2$-skeleton of the corresponding universal cover associated to the group $G$ is a proper co-H-space, see \cite[Cor. 6.4]{CLMQ}.
\end{enumerate}	
\end{remark}

\section{A counterexample to the converse of Theorem \ref{graph}}
\label{counterexample}

The purpose of this section is to show a counterexample to the converse of Theorem \ref{graph} in contrast to the situation under the quasi-isometry relation, see
\cite[Thm. 0.4]{PaWhy}).\\
\indent Let us consider the finitely presented groups $G= {\Z}_2$ and $H = {\Z} \times {\Z} \times {\Z}$ and let $Z= {\R}P^2$ and $Y$ be the $2$-skeleton of $S^1 \times S^1 \times S^1$, with $\pi_1(Z) \cong {\Z}_2$ and $\pi_1(Y) \cong H$. We now proceed to show that the infinite ended groups $G*G*G$ and $H*H$ are proper $2$-equivalent, but clearly $G$ and $H$ are not proper $2$-equivalent groups as $G$ is finite and $H$ is $1$-ended. For this, consider the $2$-dimensional CW-complexes $Y \vee Y$ and $X = Z \vee Z \vee Z$. Clearly, $\pi_1(X) \cong G*G*G$ and $\pi_1(Y \vee Y) \cong H*H$.\\
\indent The universal cover $\widetilde{X}$ can be obtained by doubling the edges (not the vertices) of the Cantor tree of degree $3$ at each vertex, and then identifying each double edge with the equator of a copy of the $2$-sphere $S^2$. One can readily check that this space is proper homotopy equivalent to the $2$-spherical object $\Gamma$ obtained by attaching a copy of $S^2$ at each vertex of the Cantor tree above. On the other hand, the universal cover $\widetilde{Y}$ can be regarded as the $2$-dimensional complex in ${\R}^3$ given by the union of the boundaries of the unit cubes whose edges are parallel to the coordinate axes and whose vertex set is ${\Z} \times {\Z} \times {\Z} \sub {\R}^3$, and hence as a (strong) proper deformation retract of ${\R}^3$ minus a countable collection of $3$-balls, one for each of these unit cubes. Thus, as observed in Remark \ref{r0}(b), $\widetilde{Y}$ has the proper homotopy type of the $2$-spherical object $\Sigma$ under ${\R}_+$ obtained by attaching a copy of $S^2$ at each non-negative integer. Moreover, by Lemma \ref{Antonio}, the proper homotopy equivalence $f : \widetilde{Y} \lga \Sigma$ can be assumed to restrict to a bijection between the $0$-skeleta $f_0 : (\widetilde{Y})^0 \lga \Sigma^0$.\\
\indent Recall from the proof of Proposition \ref{free} that the universal cover $\widetilde{Y \vee Y}$ is constructed by a tree-like arrangement of copies of $\widetilde{Y}$ as the pushout of a diagram

\[
\bigsqcup_{p \in {\N}} \widetilde{Y}_p \stackrel{i}{\longleftarrow}
{\N} \times {\N} \stackrel{\alpha}{\longrightarrow}  {\N} \times {\N} \stackrel{j}{\longrightarrow}
\bigsqcup_{r \in {\N}} \widetilde{Y}_r
\]
where $\alpha$ is a bijection determining how the copies $\widetilde{Y_p}$ are attached to the copies $\widetilde{Y_r}$ to get $\widetilde{Y \vee Y}$. Then, we make a replica of the previous pushout diagram via the bijection $f_0$ above and get a commutative diagram

\[
\xymatrix{
\ar[d]^{\bigsqcup_p f_p} \displaystyle \bigsqcup_{p \in {\N}} \widetilde{Y}_p &
\ar[l]_{i} {\N} \times {\N} \ar[d]^{(id, f_0)} \ar[r]^{\alpha} & {\N} \times {\N} \ar[d]^{(id, f_0)} \ar[r]^{j} & \displaystyle
\bigsqcup_{r \in {\N}} \widetilde{Y}_r \ar[d]^{\bigsqcup_r f_r}\\
\displaystyle \bigsqcup_{p \in {\N}} \Sigma_p  & \ar[l]_{i'} {\N} \times {\N} \ar[r]^{\alpha'} & {\N} \times {\N}
\ar[r]^{j'} & \displaystyle \bigsqcup_{ r \in {\N}} \Sigma_r
}
\]
where the $f_p$'s and $f_r$'s are copies of the map $f$, $i'$ is given by $i'(a,b) = f \circ i \circ (id, f_0)^{-1}(a,b)$, similarly $j'$, and $\alpha' = (id,f_0) \circ \alpha \circ (id,f_0)^{-1}$. As in the proof of Proposition \ref{free}, the gluing lemma \cite[Lemma I.4.9]{BQ} yields a proper homotopy equivalence between $\widetilde{Y \vee Y}$, as the pushout of the upper row, and the pushout of the lower row, say $\Delta$, which is a $2$-spherical object under a Cantor tree with two copies of the $2$-sphere $S^2$ attached at each vertex. Thus, by the classification of spherical objects in \cite[Prop. II.4.5]{BQ}, $\Delta$ has the proper homotopy type of the spherical object $\Gamma$ above, as they are both $2$-spherical objects under Cantor trees with $2$-spheres attached along each end, and hence the groups $G*G*G$ and $H*H$ are proper $2$-equivalent.

\section*{Appendix}
\label{App}

The purpose of this appendix is to give an example of an infinite ended P3R group (which is semistable at each end) so that the set of $3$-manifolds associated to it contains infinitely many different topological types in contrast to the $1$-ended case, see Remark \ref{values} (b) above.\\
\indent Let us consider the finitely presented group $G={\Z} \times {\Z}$, and let $X=S^1 \times S^1$ denote the standard $2$-dimensional CW-complexes associated to the obvious presentation. It is clear that the universal cover $\widetilde{X}={\R}^2$ is contractible and thickens to the (p.l.) $3$-manifold $M={\R}^2 \times [-1,1]$ so that the
inclusion $\widetilde{X} \hookrightarrow M$ is a proper homotopy equivalence. Observe that
$\partial M$ consists of two planes, i.e., $\mathfrak{P}(G)=2$ (see $\S \ref{App1}$). Consider the free product $G*G$ which is again P3R by \cite[Lemma 3.2]{ACLQ} and the $2$-dimensional CW-complex $P$ obtained from $X \sqcup X \sqcup I$ by identifying $0,1 \in I$ with the corresponding base point in each copy of $X$. Is it clear that $\pi_1(P) \cong G*G$.  Following the same argument as in the proof of \cite[Lemma 3.2]{ACLQ}, we next proceed to give instructions to build a family of manifolds $\widehat{M}$ associated to $G*G$ as follows.\\

\indent Consider copies $\widetilde{X}_p$ and $\widetilde{X}'_r$ of $\widetilde{X}$, as well a filtration $\{C_m\}_{m \geq 1}$ of the $3$-manifold $M$ by compact subsets. Following the same argument as in the proof of \cite[Lemma 3.2]{ACLQ}, the universal
cover $\widetilde{P}$ of $P$ is proper homotopy equivalent to a (p.l) $3$-manifold
$\widehat{M}$ obtained as the corresponding quotient space constructed from the following data:\\
\indent (i) Disjoint unions $\displaystyle \bigsqcup_{p \in {\N}} M_p$ and $\displaystyle \bigsqcup_{r \in {\N}} M'_r$ of copies of $M$, so that $\widetilde{X}_p \sub M_p$ and $\widetilde{X}'_r \sub M'_r$.\\
\indent (ii) Copies $\{C^p_m\}_{m \geq 1}$ and $\{{C'}^r_m\}_{m \geq 1}$ of $\{C_m\}_{m \geq 1}$ as filtrations for each $M_p$ and $M'_r$ respectively.\\
\indent (iii) A bijection $\varphi : {\N} \times {\N} \lga {\N} \times {\N}$ (given by the corresponding group
action of $G*G$ on $\widetilde{P}$) so that each pair $(p,q), \varphi(p,q)$ determines vertices $x_{p,q} \in \widetilde{X}_p$ and $x'_{\varphi(p,q)} \in \widetilde{X}'_r$ joined by a single copy $I_{p,q}$ of $I$ inside $\widetilde{P}$ (where $r=\pi_1(\varphi(p,q)$)).\\
\indent (iv) Functions $\widehat{m}, \widehat{m}' : {\N} \times {\N} \lga {\N}$ so that $x_{p,q} \in \widetilde{X}_p - C^p_{\widehat{m}(p,q)}$ and $x'_{\varphi(p,q)} \in \widetilde{X}'_r - {C'}^r_{\widehat{m}'(\varphi(p,q))}$, together with proper cofibrations
\[
j : {\N} \times {\N} \lga \bigsqcup_{p \in {\N}} M_p \; , j(p,q) \in \partial M_p - C^p_{\widehat{m}(p,q)}
\]
\[
j' : {\N} \times {\N} \lga \bigsqcup_{r \in {\N}} M'_r \; , j'(p,q) \in \partial M'_r -
{C'}^r_{\widehat{m}'(\varphi(p,q))} \mbox{ with } r=\pi_1(\varphi(p,q))
\]
so that $j(p,q)$ and $j'(p,q)$ are joined by paths to $x_{p,q}$ and $x'_{\varphi(p,q)}$ respectively (inside the corresponding copy of $M$).\\
\indent (iv) A disjoint union $\displaystyle \bigsqcup_{p,q \in {\N}} I_{p,q}$ of copies of the unit interval $I$, so that $0 \in I_{p,q}$ is being identified with $j(p,q) \in \partial
M_p$ and $1 \in I_{p,q}$ is being identified with $j'(p,q) \in \partial M'_r$ (with  $r=\pi_1(\varphi(p,q))$).\\
\indent (v) Finally, the corresponding $3$-manifold $\widehat{M}$ is obtained by attaching
three-dimensional $1$-handles $H_{p,q}$ to this quotient space whose cores run along each $I_{p,q}$.\\
\indent Moreover, with some additional care, the points $j(p,q)$ and $j'(p,q)$ above can be chosen so as to avoid any given subset $\mathcal S$ of the set of all plane boundary components in the original $\displaystyle \bigsqcup_{p \in {\N}} M_p \sqcup \bigsqcup_{r \in {\N}} M'_r$ satisfying that $\mathcal S$ does not contain the two boundary components of the same copy of $M$. Notice that the homeomorphism type of the $3$-manifold $\widehat{M}$ obtained as prescribed above depends somehow on how these points are disposed in the whole construction, as they determine the homeomorphism type of $\partial \widehat{M}$, following (v) above. This way, we are
choosing a particular homeomorphism type for such an $\widehat{M}$. Indeed, all planes in $\mathcal S$ remain as boundary components of $\widehat{M}$; therefore, any two subsets $\mathcal S$ and $\mathcal S'$ as above with different cardinality yield non homeomorphic $3$-manifolds as the result of the construction.

\end{document}